\documentclass[11pt, leqno]{amsart}
\usepackage{amsthm,amsmath,amssymb,amscd,verbatim}
\usepackage[curve,matrix,arrow,frame,tips]{xy}
\usepackage[colorlinks=true,pagebackref,hyperindex]{hyperref}
\usepackage{graphicx}
\usepackage{hyperref}
\usepackage{mathrsfs}
\usepackage{thmtools}

\usepackage[colorlinks=true,pagebackref,hyperindex]{hyperref}

\usepackage{color}

\theoremstyle{plain}
\newtheorem{theorem}{Theorem}[section]
\newtheorem{lemma}[theorem]{Lemma}
\newtheorem{cor}[theorem]{Corollary}
\newtheorem{prop}[theorem]{Proposition}

\newtheorem{Conj}[theorem]{Conjecture}

\newtheorem{thmx}{Theorem}

\theoremstyle{remark}

\theoremstyle{definition}

\newtheorem{defn}[theorem]{Definition}
\newtheorem{Remark}[theorem]{Remark}
\newtheorem{??}[theorem]{Question}
\newtheorem{???}[theorem]{Questions}

\numberwithin{equation}{section}
\numberwithin{theorem}{section}

\newcommand\hfld[2]{\smash{\mathop{\hbox to 10mm{\rightarrowfill}}
     \limits^{\scriptstyle#1}_{\scriptstyle#2}}}
\newcommand\hflg[2]{\smash{\mathop{\hbox to 10mm{\leftarrowfill}}
     \limits^{\scriptstyle#1}_{\scriptstyle#2}}}

\title{Dualities for root systems with automorphisms \\ and applications to non-split groups}

\author{Thomas J. Haines}

\begin{document}

\thanks{Research of T.H.~partially supported by NSF DMS-1406787} 

\date{}

\maketitle

\begin{abstract} 
This article establishes some elementary dualities for root systems with automorphisms. We give several applications to reductive groups over nonarchimedean local fields: (1) the proof of a conjecture of Pappas-Rapoport-Smithling characterizing the extremal elements of the $\{ \mu \}$-admissible sets attached to general non-split groups; (2) for quasi-split groups, a simple uniform description of the Bruhat-Tits \'{e}chelonnage root system $\Sigma_0$, the Knop root system $\widetilde{\Sigma}_0$, and the Macdonald root system $\Sigma_1$, in terms of Galois actions on the absolute roots $\Phi$; and (3) for quasi-split groups, the construction of the {\em geometric basis} of the center of a parahoric Hecke algebra, and the expression of certain important elements of the stable Bernstein center in terms of this basis. The latter gives an explicit form of the test function conjecture for general Shimura varieties with parahoric level structure.
\end{abstract}

\tableofcontents

\markboth{T. Haines}{Dualities for root systems with automorphisms and applications}

\section{Main results} \label{stmt_sec}

Suppose $\Phi \supset \Delta$ is a based root system in a real vector space $V$ which is stabilized by a finite group $I$ acting linearly on $V$. We assume $\Phi$ is reduced. There are four natural ways to construct ``$I$-fixed'' reduced based root systems associated to the $I$-action on $(V, \Phi, \Delta)$.  Namely, the operations of {\em norm} $N_I$ and {\em modified norm} $N'_I$ give rise to based reduced root systems $N_I(\Phi) \supset N_I(\Delta)$ (resp.\,$N'_I(\Phi) \supset N'_I(\Delta)$) in $V^I$. Also, the operations of {\em restriction} ${\rm res}_I$ and {\em modified restriction} ${\rm res}'_I$ give rise to based reduced root systems in $V_I$. (See Definitions \ref{norm+res_defn}, \ref{4_defn}.) These four operations yield two pairs of dual root systems (see Proposition \ref{dual_prop}). Suppose $\Phi^\vee \supset \Delta^\vee$ is the dual based root system to $\Phi \supset \Delta$. Then:

\begin{thmx} \label{thm_A} The based root system $(V_I, {\rm res}_I(\Phi^\vee), {\rm res}_I(\Delta^\vee))$ is dual to $(V^I, N'_I(\Phi), N'_I(\Delta))$, and $(V_I, {\rm res}'_I(\Phi^\vee), {\rm res}'_I(\Delta^\vee))$ is dual to $(V^I, N_I(\Phi), N_I(\Delta))$.
\end{thmx}

A basic example of the above situation comes from non-split groups over certain local fields. Let $F$ be a nonarchimedean local field and let $\breve{F}$ be the completion of its maximal unramified extension in some separable closure $\bar{F}$ of $F$. Let $I$ denote the absolute Galois group of $\breve{F}$ and let $\mathcal W_F = I \rtimes \langle \tau \rangle$ be the Weil group of $F$, where $\tau$ is a geometric Frobenius element; these groups will act through finite quotients on the root systems which follow and we use the same letters to denote these quotients. Let $G$ be a connected reductive group over $F$. We assume $G$ is quasi-split over $F$. Let $A$ be a maximal $F$-split torus in $G$, with centralizer the maximal $F$-torus $T = {\rm Cent}_G(A)$; let $B = TU$ be an $F$-rational Borel subgroup with unipotent radical $U$. Let $\Phi = \Phi(G,T)$ be the set of absolute roots of $G/\bar{F}$ with simple $B$-positive roots $\Delta \subset \Phi$. Let $\breve{\Phi}$, $\Phi_0$ be the sets of relative roots for $G/\breve{F}$ and $G/F$. Recall that $\breve{\Phi}$ and $\Phi_0$ need not be reduced.

The four reduced root systems attached to $\Phi \supset \Delta$ turn out to be related to Bruhat-Tits theory. In \cite[$\S1.4$]{BT1}, Bruhat-Tits defined the \'{e}chelonnage reduced based root systems $\breve{\Sigma}$ and $\Sigma_0$ associated to $G/\breve{F}$ and $G/F$, which play a key role in the structure theory of those groups and their affine Weyl groups (cf.\,\cite{HR08}). One can read off $\breve{\Sigma}$ and $\Sigma_0$ from the tables of \cite[$\S1.4$]{BT1} and \cite[$\S4$]{Tits}. There are also the root system $\tilde{\Sigma}_0$ of Knop \cite{Kn} and the root system $\Sigma_1$ of Macdonald \cite[$\S3$]{Mac} used in the study of the associated affine Hecke algebras with possibly unequal parameters. It does not seem to have been noticed before that, when $G/F$ is quasi-split, $\breve{\Sigma}$, $\Sigma_0$, $\widetilde{\Sigma}_0$, and $\Sigma_1$ can each be described in a simple uniform way using only the Galois actions on the absolute roots $\Phi$ (see Theorems \ref{Sigma_thm}, \ref{Sigma_0_thm}, \ref{Knop_comp_thm}, and Corollary \ref{Mac_cor}).

\begin{thmx} \label{thm_B} Suppose $G/F$ is quasi-split. Then $\breve{\Sigma}$, $\Sigma_0$, $\widetilde{\Sigma}_0$, and $\Sigma_1$ can be characterized in terms of the $\mathcal W_F$-action on the absolute roots $\Phi$, as follows:
\begin{align*}
\breve{\Sigma} &\cong N'_I(\Phi)  \\ 
\Sigma_0 &\cong {\rm res}'_\tau(\breve{\Sigma}) \\
\widetilde{\Sigma}_0 &\cong {\rm res}_\tau(\breve{\Sigma}) \\
\Sigma_1 &= {\rm res}'_\tau(\breve{\Sigma}) \cup {\rm res}_\tau(\breve{\Sigma}).
\end{align*}
\end{thmx}

Another application of the ideas behind Theorem \ref{thm_A} relates to Rapoport-Zink local models of Shimura varieties. Let $\{ \mu \}$ denote a geometric conjugacy class in the cocharacter group $X_*(G)$. In \cite{PRS}, Pappas, Rapoport, and Smithling announce two conjectures on the $\{\mu\}$-admissible subset ${\rm Adm}(\{\mu\})$ of the extended affine Weyl group $\widetilde{W} = X_*(T)_I \rtimes \breve{W}$. Their Conjecture 4.39 on vertexwise admissibility was proved in \cite{HH}. In this article we prove the other conjecture, \cite[Conj.\,4.22]{PRS},  on the extremal elements in ${\rm Adm}(\{ \mu \})$. Here we do not need to assume $G$ is quasi-split over $F$ (in any case it will automatically be quasi-split over $\breve{F}$ by Steinberg's theorem).

\begin{thmx} \label{thm_C} The extremal elements of ${\rm Adm}(\{ \mu\})$ are the images in $X_*(T)_I$ of the $\breve{F}$-rational elements $\widetilde{\Lambda}_{\{\mu \}}$ in the geometric conjugacy class $\{\mu \} \subset X_*(G)$.
\end{thmx}
See Conjecture \ref{main_conj} and Theorem \ref{PRS_thm}. This is of fundamental importance for understanding the geometry of special fibers of Rapoport-Zink local models. If $(G_{\mathbb Q_p}, \{ \mu \})$ is the $p$-adic data associated to PEL Shimura data $({\bf G}/\mathbb Q, \{ h \}, K^pK_p)$ where $K_p \subset G(\mathbb Q_p)$ is an Iwahori subgroup, and if $M_\mu$ is the associated Rapoport-Zink local model \cite{RZ}, then Theorem \ref{thm_C} ensures that we have a good understanding of the set of irreducible components in the special fiber of $M_\mu$. See $\S\ref{extremal_sec}$ for more discussion.

We now assume again that $G/F$ is quasi-split. With Theorems \ref{thm_B} and \ref{thm_C} in hand, we can construct and study a {\em geometric basis} for the center $\mathcal Z(G(F), J)$ of the parahoric Hecke algebra associated to $G(F)$ and an parahoric subgroup $J \subset G(F)$. Recall that $\mathcal W_F$ acts on the dual group $\widehat{G}$, preserving a splitting $(\widehat{B},\widehat{T}, \widehat{X})$. A dual-group consequence of Theorem \ref{thm_C} is that there is an isomorphism of $\widehat{G}^I$-modules 
$$
V_\mu^I =  \bigoplus_{\bar{\lambda} \in {\mathcal Wt}(\bar{\mu})^+}  a_{\bar{\lambda}, \mu} V_{\bar{\lambda}}
$$
for $a_{\bar{\lambda}, \mu} \in \mathbb Z_{\geq 0}$. Here $\bar{\lambda} \in X^*(\widehat{T}^I)$ is the image of $\lambda \in X^*(\widehat{T})$, $V_\mu$ and $V_{\bar{\lambda}}$ are irreducible highest weight representations of the reductive groups $\widehat{G}$ and $\widehat{G}^I$, and ${\mathcal Wt}(\bar{\mu})$ is the set of $\widehat{T}^I$-weights in $V_{\bar{\mu}}$. These extend to representations $V^I_\mu$ and $V_{\bar{\lambda},1}$ of $\widehat{G}^I \rtimes \langle \tau \rangle$ (see $\S\ref{highest_wt_sec}$ and $\S\ref{geom_basis_sec}$). The following theorem summarizes Lemma \ref{geom_basis_lem} and Theorem \ref{coeff_KL_poly}. 

\begin{thmx} \label{thm_D}
There is a basis $\{ C_{\bar{\lambda}, J}\}_{\bar{\lambda}}$ for $\mathcal Z(G(F), J)$ indexed by $\bar{\lambda} \in X^*(\widehat{T}^I)^{+, \tau}$, characterized by: 
\begin{equation*}
C_{\bar{\lambda}, J} \,\,\,\, \mbox{acts on $\pi^J$ by the scalar ${\rm tr}(s(\pi) \rtimes \tau \, | \, V_{\bar{\lambda}, 1})$}
\end{equation*}
whenever $\pi$ is an irreducible smooth representation of $G(F)$ with $\pi^J \neq 0$ and Satake parameter $s(\pi)$ \textup{(}see \cite{H15}\textup{)}. Furthermore, in terms of Bernstein functions $z_{\bar{\nu}, J} \in \mathcal Z(G(F), J)$ and certain Kazhdan-Lusztig polynomials $P_{w_{\bar{\nu}}, w_{\bar{\lambda}}}(q^{1/2})$ associated to the affine Hecke algebra with parameters ${\bf H}(\widetilde{W}^\tau, S^\tau_{\rm aff}, L)$ there is an equality 
\begin{equation} \label{C_intro_eq}
C_{\bar{\lambda}, J} = \sum_{\bar{\nu} \in {\mathcal Wt}(\bar{\lambda})^{+, \tau}} P_{w_{\bar{\nu}}, w_{\bar{\lambda}}} (1) \, z_{\bar{\nu}, J}.
\end{equation}
\end{thmx}
We call the $C_{\bar{\lambda}, J}$ the {\em geometric basis} elements for the following reason: when $F = \mathbb F_q(\!(t)\!)$ and $J = L^+P_{\bf f}(\mathbb F_q) = P_{\bf f}(\mathbb F_q[\![t]\!])$ is a very special maximal parahoric subgroup of $G(\mathbb F_q(\!(t)\!))$, $C_{\bar{\lambda},J}$ is the element of $\mathcal H(G(\mathbb F_q(\!(t)\!)), J)$ arising, via the function-sheaf dictionary, from the equivariant perverse sheaf on the affine flag variety $LG/L^+P_{\bf f}$ which corresponds to $V_{\bar{\lambda}} \in {\rm Rep}(\widehat{G}^I)$ under the geometric Satake isomorphism (see \cite{Ric2}, \cite{Zhu1}). 
Moreover, in the case where $G/\mathbb F_q(\!(t)\!)$ is {\em split} and $J = P_{\bf a}(\mathbb F_q[\![t]\!])$ is an {\em Iwahori} subgroup, we may interpret (\ref{C_intro_eq}) as the identity proved by Gaitsgory \cite{Ga} (see also \cite{HN02, GH}) 
$$
{\rm tr}^{\rm ss}(\tau \, | \, R\Psi^{M_{\mu}}) = \sum_{\lambda \preceq \mu} m_\mu(\lambda) z_\lambda,
$$
where $M_{\mu}$ is the piece of the Beilinson-Gaitsgory degeneration indexed by the dominant cocharacter $\mu$, $R\Psi^{M_\mu}$ is the (suitably normalized) complex of nearby cycles (with corresponding function ${\rm tr}^{\rm ss}(\tau \, | \, R\Psi^{M_\mu})$), and where for dominant $\lambda, \mu \in X^*(\widehat{T})$, $m_\mu(\lambda)$ is the multiplicity of the $\lambda$-weight space in $V_\mu$.

This is all connected with expectations related to the stable Bernstein center and the geometric Bernstein center of \cite{H14}. From $V_\mu \in {\rm Rep}(\widehat{G} \rtimes \mathcal W_F)$ we construct an element $Z_{V_\mu}$ of the geometric Bernstein center (see $\S\ref{some_elts_sec}$), and we can regard it as a distribution in the usual Bernstein center.  By convolving this distribution with the characteristic function $1_J$, we get an element $Z_{V_\mu} * 1_J \in \mathcal Z(G(F), J)$. For $\tau$-fixed and dominant $\bar{\lambda}$, the multiplicity space ${\mathbb H}_\mu(\bar{\lambda})$ of $V_{\bar{\lambda}}$  in $V^I_\mu|_{\widehat{G}^I}$ carries an action of $\tau$. Using this, in Theorem \ref{geom_center_thm} we give the following explicit formula in terms of the geometric basis
$$
Z_{V_\mu} * 1_J = \sum_{\bar{\lambda} \in {\mathcal Wt}(\bar{\mu})^{+, \tau}} {\rm tr}(\tau \, | \, \mathbb H_\mu(\bar{\lambda})) \, C_{\bar{\lambda}, J}.
$$
In $\S\ref{Shim_var_sec}$, we explain how this gives an explicit form of the Test Function Conjecture for Shimura varieties with parahoric level at $p$ (see \cite[$\S7$]{H14}), in the case where $G ={\bf G}_{\mathbb Q_p}$ is quasi-split over the relevant extension of $\mathbb Q_p$.

\smallskip

We now give an outline of the contents of the paper. In $\S\ref{notation_sec}$ we list some notation that will be used throughout the article. In $\S\ref{duality_notions_sec}$ we define the four operations giving ``$I$-fixed'' root systems, and prove the basic duality result in Proposition \ref{dual_prop} (Theorem \ref{thm_A}). In $\S\ref{extremal_sec}$ we prove the Pappas-Rapoport-Smithling conjecture (Theorem \ref{thm_C}) and also explain an earlier geometric approach in $\S\ref{geom_sec}$.  In $\S\ref{dual_grp_sec}$ we explain how the duality operations on root systems are related to the operation of taking $I$-fixed points in the dual group $\widehat{G}$, and we construct highest weight representations of the possibily disconnected reductive group $\widehat{G}^I$; these play a role in $\S\ref{geom+stable-center_sec}$. In section $\S\ref{BT-Knop-Mac_sec}$, we prove Theorem \ref{thm_B}, in several steps. In $\S\ref{geom+stable-center_sec}$ we prove Theorem \ref{thm_D}. This relies on the Twisted Weyl Character formula (recalled in Theorem \ref{twisted_Weyl}) and Knop's version of the Lusztig character formula (Theorem \ref{Lus_char_form}); note that Theorem \ref{thm_B} is a key ingredient used to relate these two formulas.  Also in $\S\ref{geom+stable-center_sec}$ we prove the formula for $Z_{V_\mu} * 1_J$ (Theorem \ref{geom_center_thm}).  Finally, in $\S\ref{Shim_var_sec}$, we explain the implications of this formula for the Test Function Conjecture for parahoric level \cite[$\S7$]{H14}.
\medskip

\noindent {\em Acknowledgements:} I thank Xuhua He for useful conversations around the subject of this note, and I thank Michael Rapoport and Xinwen Zhu for their remarks on an early version of the manuscript.  I am grateful to Xuhua He and Yiqiang Li for informing me of the reference \cite{Hong}. Finally, thanks go to the referee for suggestions which helped to improve the exposition.


\section{General notation} \label{notation_sec}

Given a reductive group $G$, we let $G_{\rm sc}$ denote the simply-connected cover of the derived group $G_{\rm der}$, and let $T_{\rm sc}$ denote the pull-back along $G_{\rm sc} \rightarrow G$ of a maximal torus $T$; note $T_{\rm sc}$ is a maximal torus of $G_{\rm sc}$.

For any possibly non-reduced root system $R$ in an $\mathbb R$-vector space $V$, let $R_{\rm red}$ (resp.\, $R^{\rm red}$) denote the reduced root system we get by discarding all roots of the form $2a$ (resp.\,$\frac{1}{2}a$) where $a \in R$.

For any based root system $R$, we let $R^+$ denote the positive roots.  For a lattice of (co)weights $X$, we let $X^+$ denote the subset of dominant elements.


\section{Notions of duality for root systems with automorphisms} \label{duality_notions_sec}

Let $\Phi \supset \Delta$ be a  based root system in an $\mathbb R$-vector space $V$. We assume $\Phi$ is reduced. Suppose a finite group $I \subset {\rm Aut}_\mathbb R(V)$ preserves $\Phi$ and $\Delta$.  Then we say $I$ is a group of automorphisms of the root system $(V, \Phi, \Delta)$.  Choose a positive definite symmetric bilinear form $(\cdot \, | \, \cdot)$ on $V$. We may assume it is $I$-invariant. We also use $(\cdot \, | \, \cdot)$ to identify $V$ with its $\mathbb R$-linear dual. For $v \in V - \{ 0 \}$ we define
$$
v^\vee := \frac{2v}{(v\, | \, v)}.
$$
Then for $\alpha \in \Phi$ we have the reflection $s_\alpha$ on $V$ defined by $s_\alpha(v) = v - \frac{2(\alpha | v)}{(\alpha | \alpha)} \alpha$.  Then let $W \subset {\rm Aut}_{\mathbb R}(V)$ be the finite Weyl group of $(V, \Phi, \Delta)$. Recall that $(W, S)$ is a Coxeter system with $S := \{ s_{\alpha} \, | \, \alpha \in \Delta\}$.  Note that $I$ acts on the Coxeter system $(W,S)$. Also, $(\cdot \, | \, \cdot)$ is $W \rtimes I$-invariant.

We will construct four reduced root systems associated to the $I$-action on $(V, \Phi, \Delta)$, which live naturally in the vector spaces $V_I$ or $V^I$. Sometimes it will be convenient to identify these vector spaces, using the $I$-averaging map 
\begin{align*}
V_I \, &\overset{\sim}{\rightarrow} \, V^I \\
\bar{v} ~ &\mapsto ~ v^\diamond := \frac{1}{|Iv|} \sum_{\sigma \in I} \sigma(v),
\end{align*}
where $\bar{v}$ is the image of $v \in V$ in $V_I$.

\begin{defn} \label{norm+res_defn}
\begin{enumerate}
\item[(1)] For $v \in V$, define its norm  $\displaystyle N_I(v) := \sum_{v' \in Iv} v'$.  
\item[(2)] Let ${\rm res}_I(v) = \bar{v}$ denote the image of $v$ in $V_I$.
\item[(3)] For $\alpha \in \Delta$, define its {\rm modified norm}
$$
N'_I(\alpha) = \begin{cases} N_I(\alpha), \,\,\, \mbox{if the elements in $I\alpha$ are pairwise $(\cdot \, | \,\cdot)$-orthogonal} \\ 2\,N_I(\alpha), \,\,\, \mbox{otherwise.} \end{cases}
$$
\item[(4)] For $\alpha \in \Delta$, define its {\rm modified restriction}
$$
{\rm res}'_I(\alpha) = \begin{cases} {\rm res}_I(\alpha),\,\,\, \mbox{if the elements in $I\alpha$ are pairwise $(\cdot\,|\, \cdot)$-orthogonal}, \\
2\,{\rm res}_I(\alpha), \,\,\, \mbox{otherwise.} \end{cases} 
$$
\end{enumerate}
\end{defn}
 
For $\alpha \in \Phi^+$, its $I$-average $\alpha^\diamond$ is never 0 (since $I$ permutes the simple positive roots, which are linearly independent). If $\Phi$ is the set of absolute roots for a connected reductive group over $\breve{F}$, then ${\rm res}_I(\Phi^+) \cong (\Phi^+)^\diamond$ is precisely the set of positive relative roots $\breve{\Phi}^+$ for $G/\breve{F}$ (cf.\,e.g.\,\cite[15.5.1]{Sp}). 

\begin{lemma} \label{average_lemma}  Let $\alpha \in \Delta$, and let $I\alpha$ denote the $I$-orbit of $\alpha$. Then 
\begin{equation} \label{vee-diamond}
(\alpha^\vee)^\diamond = \begin{cases} \frac{1}{|I\alpha|} \,(\alpha^\diamond)^\vee\, \,\, \, , \,\,\,\,\,\mbox{if the roots in $I\alpha$ are pairwise $(\cdot\, | \, \cdot)$-orthogonal} \\
\frac{1}{2\,|I\alpha|}\, (\alpha^\diamond)^\vee\, , \,\,\,\, \mbox{otherwise}. \end{cases}
\end{equation}
\end{lemma}

\begin{proof}
Clearly $(\alpha^\vee)^\diamond = \frac{2}{(\alpha \, | \, \alpha)}\, \alpha^\diamond = \frac{(\alpha^\diamond \, | \, \alpha^\diamond)}{(\alpha \, | \, \alpha)} (\alpha^\diamond)^\vee$.  Therefore we just need to compute the ratio $\frac{(\alpha^\diamond\, | \, \alpha^\diamond)}{(\alpha \, | \, \alpha)}$. The group $I$ permutes the irreducible subsystems in $\Phi$, each of which carries an action of its stabilizer in $I$. Also, if $v= v_1, \dots, v_k \in V$ are pairwise orthogonal vectors with $(v_i \, | \, v_i) = (v \, | \, v)$ for all $i$ and with average $v^\diamond$, then $(v^\diamond \, | \, v^\diamond) = \frac{1}{k}(v \, | \, v)$. These two remarks reduce us to considering irreducible $\Phi$, and also prove the first case of the lemma.  The second case only arises when $\Phi$ is of type $A_{2n}$, and $I$ acts via the non-trivial diagram automorphism. In that case we may assume the $I$-orbit of $\alpha$ is $\{ e_n - e_{n+1}, \, e_{n+1} - e_{n+2} \}$ and $(\cdot\, | \, \cdot)$ is the standard inner product on $\mathbb R^{2n+1}$, and the lemma follows by direct computation.
\end{proof}

\begin{defn} \label{4_defn}
We define four reduced root systems, the first two in $V^I$ and the latter two in $V_I$:
\begin{enumerate}
\item[(1)] $N_I(\Phi) := W^I N_I(\Delta)$
\item[(2)] $N'_I(\Phi) := W^I N'_I(\Delta)$
\item[(3)] ${\rm res}_I(\Phi) := W^I {\rm res}_I(\Delta)$
\item[(4)] ${\rm res}'_I(\Phi) := W^I {\rm res}'_I(\Delta)$.
\end{enumerate}
The root systems $$N_I(\Phi), \, N'_I(\Phi), \, {\rm res}_I(\Phi), \, {\rm res}'_I(\Phi)$$ come equipped with natural bases $$N_I(\Delta), \, N'_I(\Delta), \, {\rm res}_I(\Delta), \, {\rm res}'_I(\Delta),$$
and all have Weyl groups naturally identified with $W^I$. 
\end{defn}

\begin{Remark}
To show that ${\rm res}_I(\Phi)$ and ${\rm res}'_I(\Phi)$ are root systems and have Weyl groups identical to $W^I$, one may use, for instance, the argument of \cite[Lem.\,4.2]{H15}. Indeed, ${\rm res}_I(\Phi)$ (resp.\,${\rm res}'_I(\Phi)$) identifies with the set of short (resp.\,long) vectors in the set of all $I$-averages of elements of $\Phi$,  and the latter set is a possibly non-reduced root system; further, both short and long subsystems have $W^I$ as Weyl group, as explained in the proof of \cite[Lem.\,4.2]{H15}. Then by Proposition \ref{dual_prop} below, we can also show that $N_I(\Phi)$ and $N'_I(\Phi)$ are root systems and we may also identify their Weyl groups with $W^I$.
\end{Remark}

The following proposition shows that the operations ${\rm res}_I$ and $N'_I$ (resp.,\,${\rm res}'_I$ and $N_I$) are dual to each other, as claimed in Theorem \ref{thm_A}.

\begin{prop} \label{dual_prop} Let $(V, \Phi^\vee, \Delta^\vee)$ be the dual root system to $(V, \Phi, \Delta)$ defined using $(\cdot \, | \, \cdot)$, endowed with the induced action by the group $I$. Via the isomorphism $V_I \overset{\sim}{\rightarrow} V^I$, there are identifications
\begin{align*}
{\rm res}_I(\Phi^\vee)^\vee &= N'_I(\Phi) \\
{\rm res}'_I(\Phi^\vee)^\vee &= N_I(\Phi).
\end{align*}
\end{prop}

\begin{proof}
Using Lemma \ref{average_lemma}, we see that ${\rm res}_I(\alpha^\vee)^\vee = N'_I(\alpha)$ and ${\rm res}'_I(\alpha^\vee)^\vee = N_I(\alpha)$ for $\alpha \in \Delta$. The proposition follows.  
\end{proof}

\begin{Remark} \label{KS_rmk}
The essential point was already made in \cite[(1.3.9)]{KS}.
\end{Remark}

\smallskip

\noindent {\bf Example.}  Let $V = \mathbb R^{2n+1}$, let $\Phi = \{ \pm(e_i - e_j)\, | \, 1 \leq i < j \leq 2n+1\}$ be the standard Type $A_{2n}$ root system, and consider the standard simple roots $\Delta = \{ e_i - e_{i+1} \, | \, 1 \leq i \leq 2n \}$. Let $I = \langle \tau \rangle$ where $\tau$ operates by 
$$
\tau(x_1,\,\dots, x_{2n+1}) = (-x_{2n+1}, \dots, -x_1).
$$
Then $\tau$-averaging gives an isomorphism $V_\tau \overset{\sim}{\rightarrow} V^\tau = \{(x_1, \dots, x_n, 0, -x_n, \dots, -x_1) \, | \, x_i \in \mathbb R\}$. We identify $V^\tau = \mathbb R^n$ in the obvious way.  Then inside $V^\tau$ we describe the four sets of simple positive roots:
\begin{align*}
{\rm res}_\tau(\Delta) &= \{ \frac{e_1-e_2}{2}, \dots, \frac{e_{n-1}-e_n}{2}, \frac{e_n}{2}\} \hspace{1in} \mbox{Type $B_n$} \\
{\rm res}'_\tau(\Delta) &= \{ \frac{e_1 - e_2}{2}, \dots, \dfrac{e_{n-1}-e_n}{2}, e_n \} \hspace{1.04in} \mbox{Type $C_n$} \\
N_\tau(\Delta) &= \{ e_1-e_2, \dots, e_{n-1}-e_n, e_n \} \hspace{1.1in} \mbox{Type $B_n$} \\
N'_\tau(\Delta) &= \{ e_1-e_2, \dots, e_{n-1}-e_n, 2e_n \} \hspace{1.04in} \mbox{Type $C_n$}.
\end{align*}

\section{Extremal elements in admissible sets} \label{extremal_sec}

\subsection{Statement of Pappas-Rapoport-Smithling conjecture} \label{PRS_stmt_subsec}

We consider a connected reductive group $G$ over the field $\breve{F}$. Recall that $G/\breve{F}$ is quasi-split by Steinberg's theorem. Let $S$ be a maximal $\breve{F}$-split torus in $G$ with centralizer the maximal torus $T$. Let $W = W(G,T)$ be the absolute Weyl group and let $\breve{W} = W(G, S)$ be the relative Weyl group.  Let $\widetilde{W} = N_G(T)(\breve{F})/T(\breve{F})_1$ be the Iwahori-Weyl group (cf.\,e.g.\,\cite[$\S4$]{PRS}); recall $T(\breve{F})_1$ is the kernel of the Kottwitz homomorphism $\kappa_T: T(\breve{F}) \rightarrow X_*(T)_I$ of \cite[$\S7$]{Ko97}.  We fix a special vertex $o$ in the apartment $V = X_*(T)_I \otimes \mathbb R$. Given $o$, there is a natural action of $\widetilde{W}$ on the apartment such that we may identify $\breve{W}$ with the  fixer of $o$ in $\widetilde{W}$. We thus get a decomposition $\widetilde{W} = X_*(T)_I \rtimes \breve{W}$ (identifying $\widetilde{W}$ with the {\em extended affine Weyl group} over $\breve{F}$). For an element $\lambda \in X_*(T)_I$, we write $t_\lambda$ when viewing it as a ``translation'' element in $X_*(T)_I \rtimes \breve{W}$.  In general $X_*(T)_I$ might have torsion, and we write $X_*(T)_I/{\rm tors}$ for the quotient by its torsion subgroup; this quotient acts freely, by translations, on the real vector space $X_*(T)_I \otimes \mathbb R$.

The group $\widetilde{W}$ carries a Bruhat order $\leq$ in a natural way (see $\S \ref{PRS_proof_sec}$). 

Fix a geometric conjugacy class of cocharacters $\{ \mu \} \subset X_*(G)$. Let $\widetilde{\Lambda}_{\{ \mu \}}$ be those elements of $\{ \mu \}$ whose image in $X_*(T) \otimes \mathbb R$ lies in the closure of a Weyl chamber corresponding to an $\breve{F}$-rational Borel subgroup of $G$. Fix a representative $\mu \in X_*(T)$ in the geometric conjugacy class $\{ \mu \}$, with $\mu \in \widetilde{\Lambda}_{\{\mu \}}$. Thus $\widetilde{\Lambda}_{\{ \mu \}} = \breve{W} \mu$, since all $\breve{F}$-rational Borel subgroups containing $T$ are $\breve{W}$-conjugate. Denote by $\bar{\nu}$ the image of $\nu \in X_*(T)$ in $X_*(T)_I$, and let $\Lambda_{\{\mu \}}$ be the image of $\widetilde{\Lambda}_{\{ \mu \}}$. Then $\Lambda_{\{ \mu \}} = \breve{W} \bar{\mu}$.

The following conjecture of Pappas-Rapoport-Smithling \cite[Conj.\,4.22]{PRS} is one of the main motivations for this article.

\begin{Conj} \label{main_conj}
Let $\{\bar{\mu} \}$ denote the image of the $W$-conjugacy class $\{ \mu \}$ in $X_*(T)_I$. Then the set of maximal elements in $\{ \bar{\mu} \}$ with respect to the Bruhat order is precisely the set $\Lambda_{\{ \mu \} }$.
\end{Conj}

Note that Conjecture \ref{main_conj} only has content when $G/\breve{F}$ is non-split.  It arose in connection with the $\{ \mu \}$-admissible set, defined to be
\begin{equation} \label{adm0_defn}
{\rm Adm}(\{ \mu \}) = \{ w \in \widetilde{W} ~ | ~ w \leq t_{\bar{\mu}'}, \,\,\, \mbox{for some $\mu' \in \{ \mu \}$} \}.
\end{equation}
Conjecture \ref{main_conj} means that the extremal elements in this set are the translations of the form $t_{\bar{\mu}'}$ where $\bar{\mu}' \in \Lambda_{\{\mu\}}$.  It implies that the $\{ \mu \}$-admissible set as defined above coincides with the way it is defined in \cite[Def.\,4.23]{PRS}:
\begin{equation} \label{adm_defn}
{\rm Adm}(\{ \mu \}) = \{ w \in \widetilde{W} ~ | ~ w \leq t_{\bar{\mu}'}, \,\,\, \mbox{for some $\bar{\mu}' \in \Lambda_{\{ \mu \}}$} \}.
\end{equation}
The admissible set plays a role in the theory of local models of Shimura varieties and in related matters (see \cite{PRS}). In particular if $(G = {\bf G}_{\mathbb Q_p}, \{ \mu \})$ arises as the $p$-adic data associated to  Shimura data $({\bf G}/\mathbb Q, \{ h \})$ and if $M_\mu$ denotes a local model for the corresponding Shimura variety with Iwahori level structure at $p$, then the set ${\rm Adm}(\{ \mu \})$ is supposed to parametrize the Iwahori-orbits in the special fiber of $M_\mu$. Conjecture \ref{main_conj} means that $\Lambda_{\{\mu \}}$ should parametrize the irreducible components of the special fiber (see $\S\ref{geom_sec}$). For many other favorable combinatorial properties of ${\rm Adm}(\{ \mu \})$, see \cite{HH}.

In this paper we will prove Conjecture \ref{main_conj} (and thus Theorem \ref{thm_C}). In fact we will prove the following theorem, which implies Conjecture \ref{main_conj}. For $\nu_1, \nu_2 \in X_*(T)$, write $\nu_1 \preceq \nu_2$ if $\nu_2 -\nu_1$ is a sum of positive coroots in $\Phi^{\vee +}$; see $\S\ref{PRS_proof_sec}$. Define $${\mathcal Wt}(\mu) := \{ \nu \in X_*(T) ~ | ~ w\nu \preceq \mu, \,\, \forall w \in W\}$$ ($\mathcal Wt$ stands for ``weight'' and the term reflects the parallel -- established by Prop.\,\ref{on_Sigma} and Cor.\,\ref{G^I_dualroots_cor} -- with the dual group side; see $\S\ref{dual_grp_sec}$).   Note that $\{ \mu \} \subseteq {\mathcal Wt}(\mu)$.

\begin{theorem} \label{PRS_thm}
In the image $\overline{{\mathcal Wt}(\mu)} \subset X_*(T)_I$, the maximal elements with respect to the Bruhat order are precisely the elements of $\Lambda_{\{ \mu \}}$.
\end{theorem}

\subsection{A geometric approach to Conjecture \ref{main_conj}} \label{geom_sec}

Assume we are in the function field setting: $\breve{F} = k(\!(t)\!)$ and $\mathcal O_{\breve F} = k[\![t]\!]$, where $k = \overline{\mathbb F}_p$ for a prime $p$.\footnote{More generally, we may replace $k$ by any algebraically closed field of characteristic $p > 0$.} In this context, Timo Richarz \cite{Ric2} has defined ``local models'' $M_\mu$ over ${\rm Spec}(k[[t]])$ in complete generality. 

Consider a description of the special fiber $M_\mu \otimes_{\mathcal O_{\breve F}} k$ on reduced loci as the following union of Schubert varieties
\begin{equation} \label{equation}
M_\mu \otimes_{\mathcal O_{\breve F}} k = \bigcup_{w \in {\rm Adm}(\{\mu\})} S_w,
\end{equation}
where here ${\rm Adm}({\{\mu\}})$ is defined as in (\ref{adm_defn}). If such a description holds, then it is possible to deduce a proof of Conjecture \ref{main_conj}.  For example, Richarz gave such an argument for Conjecture \ref{main_conj} in \cite[Rem.\,3.13]{Ric2}, and proved the inclusion ``$\supseteq$'' of (\ref{equation}) holds in general. Further, Zhu's proof of the Coherence Conjecture \cite{Zhu2} implies that ``$\subseteq$'' holds in the case where $G/\breve{F}$ is tamely ramified ($M_\mu$ may be taken to represent his global Schubert variety). In fact, Zhu had earlier proved (\ref{equation}) in the function field case, also under some hypotheses of tame ramification for the group $G/\breve{F}$ \cite[Thm.\,3]{Zhu2}. 

The upshot is that Conjecture \ref{main_conj} was already known to hold in the function-field setting, when $G$ is tamely ramified over $\breve{F}$. It seems difficult to use the same method to go beyond the function-field setting and the case where $G$ is tamely ramified.  The advantage of the simple combinatorial proof of this article is clearly that it does not rely on any such hypotheses on the field $\breve{F}$ or on the group $G/\breve{F}$.\footnote{Subsequently, the author and Timo Richarz proved in \cite[Thm.\,6.12]{HaRi} that (\ref{equation}) always holds, with no tameness assumption. But Theorem \ref{PRS_thm} is an ingredient in this proof.}

\subsection{Proof of Theorem \ref{PRS_thm}} \label{PRS_proof_sec}

Before giving the proof, we specify how the root systems we consider arise from the group $G$. Let $\Phi$ be the set of absolute roots for $(G, T)$. Let us fix a Borel subgroup $B \supset T$ which is defined over $\breve{F}$. Write $B = T \prod_{\alpha > 0} U_\alpha$ (this determines the notion of ($B$-)positive root and dominant Weyl chamber in the vector space $X_*(T) \otimes \mathbb R$).  Let $\Delta$ denote the set of simple $B$-positive roots in $\Phi$, with simple coroots $\Delta^\vee$. Let $\breve{\Phi}$ be the set of relative roots for $(G,S)$, with set of simple $B$-positive roots denoted by $\breve{\Delta}$ and simple coroots by $\breve{\Delta}^\vee$. 

For any root system $R$ with positive roots $R^+$, we will consider the integral resp.~rational resp.~real cones $\mathbb Z_{\geq 0}R^{\vee +}$ resp.~$\mathbb Q_{\geq 0}R^{\vee +}$ resp.~$\mathbb R_{\geq 0}R^{\vee +}$ generated by the positive coroots. Of course the group generated by the former is the coroot lattice $Q^\vee(R) := \mathbb ZR^\vee$.

Recall $\widetilde{W}$ is the Iwahori-Weyl group for $G/\breve{F}$ and that it decomposes as $\widetilde{W} = X_*(T)_I \rtimes \breve{W}$ and acts in a natural way on $X_*(S) \otimes \mathbb R = X_*(T)^I \otimes \mathbb R$. Let $\breve{\bf a}$ be a base alcove in the apartment $X_*(S) \otimes \mathbb R$ of the Bruhat-Tits building $\mathcal B(G, \breve{F})$ corresponding to the torus $S$. We assume $\breve{\bf a}$ is the alcove in the $B$-dominant Weyl chamber whose closure contains the origin $o$ (our fixed special vertex). Let $\Omega_{\breve{\bf a}} \subset \widetilde{W}$ denote the subgroup stabilizing $\breve{\bf a}$.

By Bruhat-Tits theory, associated to $G_{\breve{F}}, S$ there is a reduced root system $\breve{\Sigma}$ whose affine Weyl group $W_{\rm aff}(\breve{\Sigma}) := \mathbb Z\breve{\Sigma}^\vee \rtimes \breve{W}$ may be identified with the Iwahori-Weyl group of $G_{\rm sc}/\breve{F}$ (we refer to \cite[$\S1.4$]{BT1} and \cite{HR08} for details); $\breve{\Sigma}$ is called an {\em \'{e}chelonnage root system for} $\breve{\Phi}$ by Bruhat-Tits. In general $\breve{\Sigma}$ is not contained in $\breve{\Phi}$, although every element of $\breve{\Sigma}$ is a positive multiple of an element of $\breve{\Phi}$ and vice versa. The group $W_{\rm aff}(\breve{\Sigma})$ acts on $X_*(S) \otimes \mathbb R$ in the obvious way and is a Coxeter group whose generators are the simple affine reflections $S_{\rm aff}$ through the walls of $\breve{\bf a}$. There is an isomorphism $\widetilde{W} = W_{\rm aff}(\breve{\Sigma}) \rtimes \Omega_{\breve{\bf a}}$. This is used to extend the Bruhat order $\leq$ and the length function $\ell(\cdot)$ from $W_{\rm aff}(\breve{\Sigma})$ to $\widetilde{W}$.

Recall that $\widetilde{\Lambda}_{\{\mu\}}$ denotes the elements of $\{ \mu \}$ whose image in $X_*(T) \otimes \mathbb R$ lies in the closure of a Weyl chamber corresponding to an $\breve{F}$-rational Borel subgroup of $G$. Fix a representative $\mu \in X_*(T)$ in the geometric conjugacy class $\{ \mu \}$, whose image in $X_*(T) \otimes \mathbb R$ lies in the closed Weyl chamber corresponding to $B$. Define $\Lambda_{\{\mu\}}$, ${\mathcal Wt}(\mu)$ and $\overline{{\mathcal Wt}(\mu)}$ as in subsection \ref{PRS_stmt_subsec}.

\medskip

We can now prove Theorem \ref{PRS_thm}.

\begin{proof}

Let $\lambda \in {\mathcal Wt}(\mu)$ be arbitrary. Our goal is to prove that $t_{\bar{\lambda}} \leq t_{\bar{\mu}'}$ for some $\bar{\mu}' \in \breve{W} \bar{\mu}$. The first step is to prove that $\bar{\lambda} \preceq \bar{\mu}$ in the partial ordering on $X_*(T)_I$ determined by the positive coroots $\breve{\Sigma}^{\vee +}$; we must show that $\bar{\mu} - \bar{\lambda} \in \mathbb Z_{\geq 0}\breve{\Sigma}^{\vee+}$.

We have $\mu - \lambda \in \mathbb Z_{\geq 0}\Phi^{\vee +}$. Therefore $\bar{\mu} - \bar{\lambda} \in \mathbb Z_{\geq 0}\overline{\Phi^{\vee +}}$. By \cite[Lem.\,15 and pf.~of Prop.\,13]{HR08}, we have
\begin{equation} \label{fund_equality}
Q^\vee(\Phi)_I = X_*(T_{\rm sc})_I = Q^\vee(\breve{\Sigma}).
\end{equation}
Thus we deduce that
$$
\bar{\mu} - \bar{\lambda} \in Q^\vee(\breve{\Sigma}) \cap \mathbb Z_{\geq 0}\overline{\Phi^{\vee +}}.
$$
We now claim that 
\begin{equation} \label{inclusion}
Q^\vee(\breve{\Sigma}) \cap \mathbb Z_{\geq 0}\overline{\Phi^{\vee +}} \,\, \subseteq \,\, \mathbb Z_{\geq 0}\breve{\Sigma}^{\vee +}.
\end{equation}
In order to see this inclusion, it is enough to check it after applying the map
\begin{align*}
X_*(T)_I &\rightarrow X_*(T)^I \otimes \mathbb Q \\
\bar{\nu} &\mapsto \nu^\diamond,
\end{align*}
where $\nu^\diamond$ is the average over the $I$-orbit of $\nu$. The reason it is sufficient to check (\ref{inclusion}) holds after applying $\bar{\nu} \mapsto \nu^\diamond$ is that the kernel of this map is the torsion subgroup $X_*(T)_{I, \rm tors}$ and both sides of (\ref{inclusion}) are contained in $Q^\vee(\breve{\Sigma})$, a torsion-free subgroup of $X_*(T)_I$.

Any element in $\breve{\Phi}^+$ is a positive scalar times an element in $\breve{\Sigma}^+$ (and vice versa), and thus the same holds for coroots. Hence, using Lemma \ref{average_lemma}, we see that 
\begin{equation} \label{inclusion2}
\mathbb Z_{\geq 0}(\overline{{\Phi}^{\vee +}})^\diamond \,\subset \, \mathbb Q_{\geq 0}\breve{\Phi}^{\vee +} \, \subset \,  \mathbb R_{\geq 0}\breve{\Sigma}^{\vee +}.
\end{equation} The inclusion (\ref{inclusion}) follows.

Now we have proved that $\bar{\lambda} \preceq \bar{\mu}$ in the dominance order on $X_*(T)_I$ determined by $\breve{\Sigma}^\vee$.  Let $\bar{\lambda}_{\rm dom}$ be the unique dominant element in the $\breve{W}$-orbit of $\bar{\lambda}$.  Then the same argument can be applied to show $\bar{\lambda}_{\rm dom} \preceq \bar{\mu}$. 

It follows that $t_{\bar{\lambda}_{\rm dom}} \leq t_{\bar{\mu}}$ in the Bruhat order. This is explained in \cite[Rem.\,4.17]{PRS} as a consequence of \cite[Cor.\,1.8]{Ric1}, but it can also be seen as a consequence of \cite[Lem.\,3.8 and proof of Prop.\,3.5]{Rap}. 

It is well-known that $\ell(t_{\bar{\lambda}}) = \ell(t_{\bar{\lambda}_{\rm dom}})$. 

\begin{lemma} \label{lemma2}
Let $\bar{\nu} \in X_*(T)_I$. If $x \in \widetilde{W}$ and $s \in S_{\rm aff}$ have $\ell(x) = \ell(sxs)$, and if $x \leq t_{\bar{\nu}}$, then there exists $\bar{\nu}' \in \breve{W}\bar{\nu}$ such that $sxs \leq t_{\bar{\nu}'}$.
\end{lemma}

\begin{proof}
The proof is the same as that of \cite[Lemma 4.5]{H01}.
\end{proof}

Since $\bar{\lambda}$ results from $\bar{\lambda}_{\rm dom}$ by a sequence of length-preserving conjugations by elements $s \in S_{\rm aff}$, repeated use of Lemma \ref{lemma2} shows that $t_{\bar{\lambda}} \leq t_{\bar{\mu}'}$ for some $\bar{\mu}' \in \breve{W}\bar{\mu}$.  This completes the proof of Theorem \ref{PRS_thm}.
\end{proof}

\section{Automorphisms and dual groups} \label{dual_grp_sec}

\subsection{On root systems of fixed points in dual groups}

We continue to assume $G$ is an arbitrary connected reductive group over $\breve{F}$ (which is automatically quasi-split). 

Let $\widehat{G}$ be the complex dual group of $G$. Since $G/\breve{F}$ is quasi-split, associated to $G$ there is an $I$-fixed splitting $(T, B, X = \sum_{\alpha \in \Delta} X_\alpha)$.  By construction of the $L$-group (cf.\,\cite[$\S 1$]{Ko84}), $\widehat{G}$ carries an action by $I$ which factors through a finite quotient and fixes a splitting of the form $(\widehat{T}, \widehat{B}, \widehat{X})$.  Since $G/\breve{F}$ is quasi-split, the $I$-action on $\widehat{T}$ inherited from that on $\widehat{G}$ agrees with that derived from the $I$-action on $X_*(T) = X^*(\widehat{T})$ (comp.\,\cite[$\S5$]{H15}). We can write $\widehat{B} = \widehat{T} \prod_{\alpha > 0} U_{\alpha^\vee}$ and $\widehat{X} = \sum_{\alpha \in \Delta} X_{\alpha^\vee}$. 

For any possibly non-reduced root system $R$, recall the reduced root systems $R_{\rm red}$ and $R^{\rm red}$ defined in $\S\ref{notation_sec}$. 

\begin{prop} \label{IMRN_prop}
The group $H := \widehat{G}^I$ is reductive and its neutral component $H^\circ = \widehat{G}^{I, \circ}$ is equipped with the splitting $(T_H, B_H, X_H) := (\widehat{T}^{I, \circ}, \widehat{B}^{I, \circ}, \widehat{X})$.  The maximal torus $T_H$ is the neutral component of the diagonalizable subgroup $\widehat{T}^I \subset \widehat{G}^I$. Moreover, the root system for $(H, T_H, B_H)$ is isomorphic to $[(\Phi^\vee)^\diamond]_{\rm red}$ and has $\{ (\alpha^\vee)^\diamond ~ | ~ \alpha \in \Delta \}$ as its simple positive roots. 
\end{prop}

\begin{proof} Most of this is covered by the statement of \cite[Prop.\,4.1]{H15}; the last assertion was established during the proof of that proposition.
\end{proof}

\begin{prop} \label{on_Sigma}
There is an equality of based root systems
$$[(\Phi^\vee)^\diamond]_{\rm red} = \breve{\Sigma}^\vee.$$
\end{prop}

\begin{proof}
By (\ref{fund_equality}), we have $(\mathbb Z\Phi^\vee)_I = \mathbb Z\breve{\Sigma}^\vee$.  Applying $\nu \mapsto \nu^\diamond$ we get the equality
$$
\mathbb Z(\Phi^\vee)^\diamond = \mathbb Z\breve{\Sigma}^\vee
$$
of lattices in $X_*(T)^I \otimes \mathbb Q$. The two lattices have identical monoids of positive cones as
\begin{equation} \label{pos_cones}
\mathbb Z_{\geq 0}(\Phi^\vee)^{\diamond, +} \, = \, \mathbb Z(\Phi^\vee)^\diamond \cap \mathbb R_{\geq 0} \breve{\Phi}^{\vee +} \, = \,\mathbb Z\breve{\Sigma}^\vee \cap \mathbb R_{\geq 0}\breve{\Phi}^{\vee +} \, = \, \mathbb Z_{\geq 0}\breve{\Sigma}^{\vee +}.
\end{equation}
Therefore the minimal positive elements of these monoids -- i.e., the two bases of the two root systems $[(\Phi^\vee)^\diamond]_{\rm red}$ and $\breve{\Sigma}^\vee$  -- must coincide, and thus the based root systems coincide.
\end{proof}

\begin{cor} \label{G^I_dualroots_cor}
The root system for $(\widehat{G}^{I, \circ}, \widehat{T}^{I, \circ})$ is the root system $X_*(T)_I/{\rm tors} \supset \breve{\Sigma}^\vee$ and in particular is of type dual to $\breve{\Sigma}$.
\end{cor}

\begin{Remark} \label{point_elided_rem}
Corollary \ref{G^I_dualroots_cor} is used implicitly in the works of Zhu \cite{Zhu1} and Richarz \cite{Ric2} on the geometric Satake isomorphism for certain non-split groups. For example, see \cite[Lem.\,4.10]{Zhu1} and \cite[Def.\,A.7]{Ric2}.
\end{Remark}

\subsection{On highest weight representations} \label{highest_wt_sec}

Under the identification $X_*(T) = X^*(\widehat{T})$, we may view ${\mathcal Wt}(\mu)$ as weights of $\widehat{T}$. Let $V_\mu$ denote the irreducible representation of $\widehat{G}$ with highest weight $\mu$.    Then ${\mathcal Wt}(\mu)$ is precisely the set of $\widehat{T}$-weights of $V_\mu$.  It follows that $\overline{\mathcal Wt(\mu)}$ is precisely the set of $\widehat{T}^I$-weights of the image $V_\mu|_{\widehat{G}^I}$ of $V_\mu$ under the restriction map ${\rm Rep}(\widehat{G}) \rightarrow {\rm Rep}(\widehat{G}^I)$; since $\widehat{T}^I$ is a diagonalizable group, weights for it exist.

By (\ref{pos_cones}), the natural partial orders $\preceq$ on $\overline{\mathcal Wt(\mu)}$ and on the weights of $V_{\mu}|_{\widehat{G}^I}$ coincide.

Let $\bar{\nu}^\flat$ denote the image of $\bar{\nu}$ in $X_*(T)_I/{\rm tors}$. Let $X_*(T)_I^+$ and $X_*(T)^{\flat, +}_I$ be the set of $B$-dominant elements in $X_*(T)_I$ and $X_*(T)^\flat_I$. Let $H := \widehat{G}^{I}$ and $H^\circ := \widehat{G}^{I, \circ}$.

The following lemma provides the theorem of the highest weight for the possibly disconnected reductive group $H$. Some of this was done, in a different way, by Zhu \cite[Lem.\,4.10]{Zhu1}.

\begin{lemma} \label{ht_wt_lem}
\begin{enumerate}
\item[(1)] Every irreducible representation of $H$ is a highest weight representation $V_{\bar{\lambda}}$ for some $\bar{\lambda} \in X_*(T)^+_I$.
\item[(2)] Let $V_{\bar{\lambda}}$ be an irreducible representation of $H$ with highest weight $\bar{\lambda}$. Then $V_{\bar{\lambda}}|_{H^\circ}$ is irreducible and has highest weight $\bar{\lambda}^\flat$. 
\item[(3)] Every irreducible representation of $H^\circ$ is of the form $V_{\bar{\lambda}}|_{H^\circ}$ for some $\bar{\lambda} \in X_*(T)^+_I$.
\item[(4)] For $\bar{\nu} \in X_*(T)^+_I = X^*(\widehat{T}^I)^+$ with image $\bar{\nu}^\flat \in X_*(T)_I/{\rm tors} = X^*(\widehat{T}^{I, \circ})$, and for $V_{\bar{\lambda}^\flat} = V_{\bar{\lambda}}|_{\widehat{G}^{I,\circ}}$, we have
\begin{equation} \label{wtspace=wtspace^flat}
V_{\bar{\lambda}}(\bar{\nu})|_{\widehat{T}^{I, \circ}}  = V_{\bar{\lambda}^\flat}(\bar{\nu}^\flat).
\end{equation}
\end{enumerate}
\end{lemma}

\begin{proof} We deduce this from the theorem of the highest weight for the connected reductive group $H^\circ$, as follows. By \cite[$\S4,5$]{H15}, we have $\widehat{G}^I = Z(\widehat{G})^I \cdot \widehat{G}^{I,\circ}$ and $\widehat{G}^{I,\circ} \cap \widehat{T}^I = \widehat{T}^{I, \circ}$. It follows that $Z(\widehat{G}^I) = Z(\widehat{G})^I \cdot Z(\widehat{G}^{I,\circ})$ and
\begin{align} \label{H=ZH0}
H &= Z(H)\cdot H^\circ  \notag\\
\widehat{T}^I &= Z(H) \cdot \widehat{T}^{I,\circ} \\
Z(H^\circ) &\subset \widehat{T}^{I,\circ}. \notag
\end{align}
Suppose $(V, \rho^\circ)$ is an irreducible representation of $H^\circ$. This is a highest weight representation for some element $\bar{\lambda}^\flat \in X_*(T)_I^{\flat, +} = X^*(\widehat{T}^{I, \circ})^+$, the image of an element $\bar{\lambda} \in X_*(T)_I^+ = X^*(\widehat{T}^I)^+$. The restriction of $\rho^\circ$ to $Z(H) \cap H^\circ = Z(H^\circ)$ acts through the central character, which is the restriction of $\bar{\lambda}^\flat$ to $Z(H^\circ)$. Using (\ref{H=ZH0}), we can extend $\rho^\circ$ to a representation $(V, \rho)$ of $H$ by setting
$$\rho(zh^\circ) = \bar{\lambda}(z) \,\rho^\circ(h^\circ),$$ 
where $z \in Z(H)$ and $h^\circ \in H^\circ$. This is irreducible and has highest weight $\bar{\lambda}$.  This constructs all the irreducible highest weight representations $(V_{\bar{\lambda}}, \rho_{\bar{\lambda}})$ of $H$, and proves (3). Clearly (\ref{H=ZH0}) implies (2).  If $(V, \rho)$ is any irreducible representation of $H$, then it has some maximal weight, say $\bar{\lambda} \in X^*(\widehat{T})^+$. By (\ref{H=ZH0}) the restriction $V|_{H^\circ}$ is irreducible, has maximal weight $\bar{\lambda}^\flat$, and hence is the highest weight representation for $\bar{\lambda}^\flat$. But then $V_{\bar{\lambda}}|_{H^\circ} \cong V|_{H^\circ}$, and as the central characters of $V$ and $V_{\bar{\lambda}}$ are both $\bar{\lambda}|_{Z(H)}$, (\ref{H=ZH0}) shows that $V \cong V_{\bar{\lambda}}$, and hence (1). Part (4) is straightforward.
\end{proof}

\section{Bruhat-Tits, Knop, and Macdonald root systems} \label{BT-Knop-Mac_sec}


\subsection{The \'{e}chelonnage root system $\breve{\Sigma}$} \label{breveSigma_subsec}

Recall ($\S\ref{PRS_proof_sec}$) the reduced root system $\breve{\Sigma}$ which is attached to a connected reductive group $G$ over $\breve{F}$. We may rephrase  Proposition \ref{on_Sigma} as follows. 

\begin{theorem} \label{Sigma_thm}
The Bruhat-Tits root system $\breve{\Sigma}$ attached to a group $G/\breve{F}$ with absolute root system $\Phi$ can be characterized by the identifications
\begin{align*}
{\rm res}_I(\Phi^\vee) &= \breve{\Sigma}^\vee \\
 N'_I(\Phi) &\cong \breve{\Sigma}.
\end{align*}
\end{theorem}

\begin{proof}
The equality in the first line follows from Proposition \ref{on_Sigma}, and this implies the isomorphism in the second line by Proposition \ref{dual_prop}.
\end{proof}

\begin{Remark}
The set $\breve{\Sigma}$ has previously been described in terms of $\breve{\Phi}$ instead of $\Phi$, in a non-uniform way.  For example, in the work of Prasad-Raghunathan \cite[$\S2.8$]{PrRa} (cf.\,\cite[4.15]{PRS}), one finds the following description of $\breve{\Sigma}$ for $G/\breve{F}$ simply connected and absolutely simple:
\begin{itemize}
\item If $G$ is split, then $\breve{\Sigma} = \breve{\Phi}$;
\item If $G$ is non-split and $\breve{\Phi}$ is reduced, then $\breve{\Sigma} \cong \breve{\Phi}^\vee = \{ \frac{2\alpha}{(\alpha \, | \, \alpha)}  ~ | ~ \alpha \in \breve{\Phi} \}$;
\item If $\breve{\Phi}$ is non-reduced then $\breve{\Sigma}$ is the subset of roots $\alpha \in \breve{\Phi}$ for which $2\alpha \notin \breve{\Phi}$.
\end{itemize}
One can use Theorem \ref{Sigma_thm} together with Lemma \ref{average_lemma} to give another proof of the above description of Prasad-Raghunathan.
\end{Remark}

\subsection{The \'{e}chelonnage root system ${\Sigma}_0$} \label{Sigma0_subsec}

For the remainder of the article, we assume $G$ is a quasi-split group over the field $F$. Let $A$ be a maximal $F$-split torus in $G$ with centralizer the $F$-rational torus $T$. Fix an $F$-rational Borel subgroup $B = TU$ containing $T$. Let $S \subset T$ be the maximal $\breve{F}$-split subtorus. 

Let $\mathcal W_F$ be the Weil group of the nonarchimedean local field $F$, and write $I \subset \mathcal W_F$ for the inertia subgroup; recall $I$ is isomorphic to the absolute Galois group of $\breve{F}$. Let $\tau \in \mathcal W_F$ denote a geometric Frobenius element. 

We have the relative Weyl group $W^I = \breve{W}$ for $G/\breve{F}$, and the relative Weyl group $W_0 = (\breve{W})^{\langle \tau \rangle}$ for $G/F$.  See \cite[Prop.\,4.1]{H15}. Let $\Phi_0 \subset X^*(S)^\tau \otimes \mathbb R = X^*(A) \otimes \mathbb R$ be the set of relative roots for $G/F$, in other words, the set of $\mathcal W_F$-averages of the elements of $\Phi$, as well as the set of $\tau$-averages of elements of $\breve{\Phi}$.  

We fix once and for all a {\em very special} vertex $o$ in the closure of the alcove $\breve{\bf a}$ in the apartment $X_*(S)^\tau \otimes \mathbb R$ for $G/F$.  This means that $o$ is also a special vertex in the apartment $X_*(S) \otimes \mathbb R$ for $G/\breve{F}$; very special vertices exist for quasi-split $G/F$ by \cite[$\S6$]{Zhu1}.  If $G/F$ is unramified, a vertex $o$ is very special if and only if it is hyperspecial.  Recall $\widetilde{W}$ acts on the apartment $X_*(S) \otimes \mathbb R$, and we may view $\breve{W}$ (resp. $W_0 = \breve{W}^\tau$) as the subgroup of $\widetilde{W}$ (resp. $\widetilde{W}^\tau$) fixing $o$. According to Bruhat-Tits theory (\cite[$\S1.4$]{BT1}) to $\Phi_0$ we associate a reduced root system $\Sigma_0$ such that, as in $\S\ref{PRS_proof_sec}$, $W_{\rm aff}(\Sigma_0) := \mathbb Z\Sigma_0^\vee \rtimes W_0$ is identified with the Iwahori-Weyl group $\widetilde{W}(G_{\rm sc}/F)$ of $G_{\rm sc}/F$. Given $o$, we get decompositions of the Iwahori-Weyl groups 
\begin{align*}
\widetilde{W}(G_{\rm sc}/\breve{F}) &= \mathbb Z\breve{\Sigma}^\vee \rtimes \breve{W} \\
\widetilde{W}(G_{\rm sc}/F) &= \mathbb Z\Sigma^\vee_0 \rtimes W_0
\end{align*} 
where the first decomposition is $\tau$-equivariant; see $\S\ref{PRS_proof_sec}$.  We may assume the alcove $\breve{\bf a}$ from $\S\ref{PRS_proof_sec}$ is the $\tau$-invariant alcove of the apartment $X_*(S) \otimes \mathbb R$ corresponding to an alcove ${\bf a} \subset X_*(A) \otimes \mathbb R$ whose closure contains $o$. Then it is known thanks to T.\,Richarz that $$\widetilde{W}(G_{\rm sc}/\breve{F})^\tau = \widetilde{W}(G_{\rm sc}/F);$$ see \cite[Lem.\,11.3.1]{H14}. Thus by comparing $\tau$-fixed points, we deduce
$\mathbb Z \Sigma_0^\vee = \big(\mathbb Z\breve{\Sigma}^\vee \big)^\tau$, and hence (comparing bases as in Proposition \ref{on_Sigma})
\begin{equation} \label{echelonnage_key}
\Sigma_0^\vee = N_\tau(\breve{\Sigma}^\vee).
\end{equation}
Now using Proposition \ref{dual_prop} we obtain the following:

\begin{theorem} \label{Sigma_0_thm}
For a quasi-split group $G/F$ giving rise to $\breve{\Phi}$ and $\breve{\Sigma}$ as in $\S\ref{breveSigma_subsec}$, the \'{e}chelonnage root system $\Sigma_0$ attached to $G/F$ can be characterized by the identities
\begin{align*}
\Sigma_0^\vee &= N_\tau(\breve{\Sigma}^\vee) \\
\Sigma_0 &\cong {\rm res}'_\tau(\breve{\Sigma}).
\end{align*}
\end{theorem}

\begin{Remark}
Note that the descent $\Phi \leadsto \breve{\Sigma}$ is via the operation $N'_I$, whereas the descent $\breve{\Sigma} \leadsto \Sigma_0$ is via ${\rm res}'_\tau$.  This should be contrasted with the relation between $\Phi$ and the Bruhat-Tits affine roots:  $\Phi_{\rm aff}(G/\breve{F})$ is defined using the internal structure of $G/\breve{F}$ and is not transparent from $\Phi$, and the descent $\Phi_{\rm aff}(G/\breve{F}) \leadsto \Phi_{\rm aff}(G/F)$ is via a restriction operation analogous to ${\rm res}_\tau$ (see \cite[1.10.1]{Tits}).
\end{Remark}

\begin{Remark}
Suppose $G/F$ is {\em not} quasi-split. Then we can prove a weaker statement related to Theorem \ref{Sigma_0_thm}. Note that $\tau$ still permutes the simple affine roots $\breve{\Delta}_{\rm aff}$ and the affine roots $\breve{\Sigma}_{\rm aff} = \breve{\Sigma} \times \mathbb Z$ associated to a $\tau$-stable alcove $\breve{\bf a}$ in the apartment associated to an $F$-rational maximal $\breve{F}$-split torus $S \subset G$. Let $\breve{S}_{\rm aff}$ denote the simple affine reflections in $\breve{W}_{\rm aff} = \widetilde{W}(G_{\rm sc}/\breve{F})$ and let $(\breve{S}_{\rm aff}/\tau)_{< \infty}$ denote the set of $\langle \tau \rangle$-orbits $\pi$ in $\breve{S}_{\rm aff}$ such that the parabolic subgroup $\breve{W}_{\pi} \subset \breve{W}_{\rm aff}$ generated by all $s \in \pi$ is {\em finite}. For each $\pi \in (S_{\rm aff}/\tau)_{< \infty}$, define $w_{\pi} \in \breve{W}_{\rm aff}^\tau = W_{\rm aff}$ to be the longest element in $\breve{W}_\pi$.  Then $(W_{\rm aff}, \{w_\pi\}_{\pi \in (\breve{S}_{\rm aff}/\tau)_{< \infty}})$ is a Coxeter group. Further, if $\breve{\ell} : \breve{W}_{\rm aff} \rightarrow \mathbb Z_{\geq 0}$ is the length function on $(\breve{W}_{\rm aff}, \breve{S}_{\rm aff})$, and if $w = w_{\pi_1} \cdots w_{\pi_r}$ is a reduced expression for $w \in W_{\rm aff}$, then 
$$
\breve{\ell}(w) = \sum_{i=1}^r \breve{\ell}(w_{\pi_i}).
$$
(For proofs, see \cite{Mic}.) Moreover, the simple affine roots $\Delta_{0, \rm aff} \subset \Sigma_0 \rtimes \mathbb Z$ are the (minimally positive on $\breve{\bf a}^\tau$) affine roots corresponding to the reflections $w_\pi$, $\pi \in (\breve{S}_{\rm aff}/\tau)_{< \infty}$.  (Note that it is possible to have $(\breve{S}_{\rm aff}/\tau)_{< \infty}  = \emptyset$.)
\end{Remark}

\subsection{The Knop root system $\widetilde{\Sigma}_0$} \label{Knop_subsec}

Now we let $L: S_{\rm aff} \rightarrow \mathbb Z$ be the system of parameters associated to the group $G(F)$. More precisely, let $P_{\bf a}(\mathcal O_F) \subset G(F)$ (resp.\,$P_o(\mathcal O_F) \subset G(F)$) be the Iwahori (resp.\,very special parahoric) subgroup corresponding to ${\bf a}$ (resp.\,$o$). We fix a set-theoretic embedding $\widetilde{W}^\tau = X_*(T)_I^\tau \rtimes W_0 \hookrightarrow G(F)$ by sending $w \in W_0$ to any lift in $N_G(A)(F) \cap P_{o}(\mathcal O_F)$ and by mapping $\lambda \in X_*(T)^\tau_I$ to an element $a_\lambda \in T(F)$ such that $\kappa_T(a_\lambda) = -\lambda$. (Here $\kappa_T$ is the Kottwitz homomorphism $\kappa_T : T(\breve{F}) \rightarrow X_*(T)_I$ normalized as in \cite[$\S7$]{Ko97}; for instance, if $T = \mathbb G_m$, and $\varpi$ is a uniformizer of $F$, then $\kappa_T(\varpi) = 1 \in \mathbb Z$.) 

Let $J = P_{\bf a}(\mathcal O_F)$, an Iwahori subgroup of $G(F)$, and set $q = \#(\mathcal O_F/\varpi)$. The Iwahori-Hecke algebra $\mathcal H(G(F), J) = C^\infty_c(J\backslash G(F)/J)$ is a convolution algebra, defined using the Haar measure $dx_J$ with $\int_J dx_J = 1$. It is a specialization of the affine Hecke algebra ${\bf  H}(\widetilde{W}^\tau, S^\tau_{\rm aff}, L)$
$$
\mathcal H(G(F), J) \, \cong \, {\bf H}(\widetilde{W}^\tau, S^\tau_{\rm aff}, L) \otimes_{\mathbb Z[v,v^{-1}]} \mathbb C
$$
with respect to $v \mapsto q^{1/2}$. Here $S^\tau_{\rm aff} \subset {\rm Aut}_\mathbb R(X_*(A) \otimes \mathbb R)$ is the set of affine reflections through the walls of ${\bf a}$, and $L : S^\tau_{\rm aff} \rightarrow \mathbb Z$ is a system of (possibly unequal) parameters defined by
$q^{L(s)} = \int_{JsJ} dx_J$, for $s \in S^\tau_{\rm aff}$. In particular, if $T_w := 1_{JwJ}$ for $w \in \widetilde{W}^\tau$, then
$$
T^2_s = (q^{L(s)}\!-\!1)T_s + q^{L(s)} T_e.
$$

Let $\Pi_0 \subset \Sigma_0$ be the set of simple roots.  For any {\em indivisible} root $a' \in \Phi_0$ corresponding to a root $a \in \Sigma_0$, there is a unipotent subgroup $U_{a'} \subset G$ (having Lie algebra $\mathfrak g_{a'} + \mathfrak g_{2a'}$, where $\mathfrak g_{2a'} = 0$ if $2a' \notin \Phi_0$) and subgroups $U_{a +i} \subset U_{a'}(F)$, which enter into the definition of the Bruhat-Tits affine roots $\Phi_{\rm aff}$ (cf.\,\cite[0.17, $\S4$]{Land}).  

We can now describe the above parameters in the notation of Macdonald \cite[$\S3.1$]{Mac} as
$$
q^{L(s)} = q_{a/2} q_a = q_{a+1} := [U_{a}: U_{a-1}]
$$
where $s = s_a$ for $a \in \Pi_0$. Further, we have
$$
q^{L(s_0)} = q_{-\tilde{a}+1+1} := [U_{-\tilde{a}+1} : U_{-\tilde{a}}]
$$
if $s_0$ is the simple reflection corresponding to a simple affine root $-\tilde{a}+1$.

We repeat some definitions from \cite{Kn}. We say $a \in \Pi_0$ is {\em special} if it is the long simple root in a component of type $C_n$ (setting $C_1 = A_1$), and has the property that $L(s_a) \neq L(s_0)$, where $s_0$ is the reflection for the simple affine root $-\tilde{a} + 1$ in that component.  

\begin{Remark} \label{q_a_rmk} Noting that in a component of type $C_n$ with simple positive roots $$e_1-e_2, \dots, e_{n-1}-e_n, 2e_n$$ and $\tilde{a} = 2e_1$, we have $q_{-2e_1+1+1} = q_{-2e_1} = q_{2e_n}$, and hence for $a = 2e_n$, $L(s_a) \neq L(s_0)$ is equivalent to $q_{a+1} \neq q_a$.
\end{Remark}

\begin{defn} \label{Knop_def} Following \cite[(4.1.4)]{Kn} we introduce the root system $\widetilde{\Sigma}_0$ spanned by the set $\widetilde{\Pi}_0 := \{ \epsilon(a)a \, | \, a \in \Pi_0\}$, where for $a \in \Pi_0$, we define
$$
\epsilon(a) = \begin{cases} \frac{1}{2}, \,\,\, \mbox{if $a$ is special} \\ 1, \,\,\,\, \mbox{otherwise}. \end{cases}
$$ 
\end{defn}

The following theorem should be compared with Theorem \ref{Sigma_0_thm}.

\begin{theorem} \label{Knop_comp_thm} For $G/F$ quasi-split, the root system $\widetilde{\Sigma}_0$ can be characterized by the isomorphisms
\begin{align*}
\widetilde{\Sigma}_0 &\cong {\rm res}_\tau(\breve{\Sigma}) \\
\widetilde{\Sigma}^\vee_0 &\cong N'_\tau(\breve{\Sigma}^\vee). 
\end{align*}
\end{theorem}

\begin{proof}
Our point of departure is the identification of Theorem \ref{Sigma_0_thm}: ${\rm res}'_\tau(\breve{\Pi}) = \Pi_0$, where $\breve{\Pi} \subset \breve{\Sigma}$ and $\Pi_0 \subset \Sigma_0$ are the simple roots. 

Now, to construct $\tilde{\Pi}_0$ from $\Pi_0$, we replace $a \in \Pi_0$ with $\frac{1}{2}a$ precisely when $a$ is special, that is, the long simple root in a component $C$ of type $C_n$ such that $L(s_a) \neq L(s_0)$.  Write $a = {\rm res}'_\tau(\breve{a})$ for some $\breve{a} \in \breve{\Pi}$. We claim $a$ is special if and only if
\begin{enumerate}
\item[(i)] $\breve{a}$ belongs to a component $\breve{C}$ in a union of components of type $A_{2n}$ in $\breve{\Pi}$ which are permuted by $\langle \tau \rangle$, such that
\item[(ii)] some power $\tau^i$ stabilizes and acts non-trivially on $\breve{C}$, and
\item[(iii)] in $\breve{C} = A_{2n}$, $\breve{a}$ is one of the simple roots $e_{n}-e_{n+1}$ or $e_{n+1} -e_{n+2}=\tau^i(e_n-e_{n+1})$. 
\end{enumerate}
Indeed, suppose $\breve{a}$ satisfies (i-iii). Renaming $\tau^i$ as $\tau$, we may assume $\tau$ acts (nontrivially) on $\breve{C} = A_{2n}$ and $\breve{a} = e_n- e_{n+1}$. Then $a = 2e_n$ is the long simple root in a component $C$ of type $C_n$. The group $G_{{\rm sc}, \breve{F}}$ is then isomorphic to ${\rm SL}_{2n+1}$ ($\breve{\Sigma}$ has type $A_{2n}$ implies $\breve{\Phi}$ has type $A_{2n}$ which implies $G_{{\rm sc}, \breve{F}}$ is split). The group $G_{\rm sc}/F$ is therefore a non-split unramified group which is a form of ${\rm SL}_{2n+1}$ and is nonresidually split since $\tau$ does not act trivially on $\breve{\Delta}$ (cf.\,\cite[1.11]{Tits}). Therefore $G_{\rm sc}/F$ is the group named $^2A'_{2n}$ in the first line of the table in \cite[p.63]{Tits}, in other words, an unramified ${\rm SU}(2n+1)$. The same table shows that $L(s_a) = 3$ and $L(s_0) = 1$. It follows that $a$ is special.

Conversely, assume $a$ is special, in a component $C$; clearly $C = C_n$ and $a$ is the long simple root. Then the component $\breve{C}$ containing $\breve{a}$ is one of a union permuted by $\langle \tau \rangle$, and let $\langle \tau^i \rangle$ be the stabilizer of $\breve{C}$.  Replacing $\tau^i$ by $\tau$, we may assume $G_{\rm sc}/\breve{F}$ is simple with $\breve{\Pi} = \breve{C}$ and $\Pi_0 = C$. We need to understand which simple relative root systems $\Phi_0$ can correspond to a {\em quasi-split} group $G_{\rm sc}/F$ with $\Sigma_0$ of type $C_n$ (where $C_1 := A_1$). The tables in \cite[4.2, 4.3]{Tits} show that the only possibilities for $\Phi_0$ are named the following:
\begin{enumerate}
\item[(a)] {\em residually split cases}: $A_1$, $C_n$, $C$-$B_n$, $C$-$BC_n$ ($n \geq 2$), and $C$-$BC_1$;
\item[(b)] {\em nonresidually split cases}: $^2A'_{2n-1}$ ($n \geq 2$) and $^2A'_{2n}$ ($n \geq 1$). 
\end{enumerate}
Note that the nonresidually split $\Phi_0$ with $\Sigma_0$ of type $C_n$ which are not listed in (b) (comprising 18 cases, ranging from $\Phi_0 = \,^dA_{2d-1}$ ($d \geq 2$) to $\Phi_0 =  \, ^4D_5$) correspond to $G_{\rm sc}/F$ which are visibly non-quasi-split, since $\tau$ does not fix the simple affine root for $\breve{\Delta}$.  

Now among these, we need to understand which contain $a$ as a special root. In all cases of (a), and in the cases $^2A'_{2n-1}$ of (b), the tables show $L(s_a) = 1 = L(s_0)$; these are ruled out since $a$ is special.  The only remaining possibility is $\Phi_0 = \, ^2A'_{2n}$. As seen above, this group has $\breve{\Sigma}$ of type $A_{2n}$ and $G_{{\rm sc}, \breve{F}} = {\rm SL}_{2n+1}$, with $\tau$ acting nontrivially.  Since $a$ is special, we must have $\breve{a} \in \{ e_n-e_{n+1}, e_{n+1}-e_{n+2} \}$ and these two roots are permuted by $\tau$.  (This case is really allowed, since the table shows $L(s_a) = 3 \neq 1 = L(s_0)$.) Thus $\breve{a}$ satisfies (i-iii). The claim is proved.

On the other hand, to get ${\rm res}_\tau(\breve{\Pi})$ from ${\rm res}'_\tau(\breve{\Pi})$, we replace $a = {\rm res}'_\tau(\breve{a})$ for $\breve{a} \in \breve{\Pi}$ with $\frac{1}{2}a$ precisely when  the elements of $\langle \tau \rangle\breve{a}$ are not pairwise orthogonal. But this happens exactly when $\breve{a}$ satisfies (i-iii) above.

We have proved that $\tilde{\Pi}_0$ and ${\rm res}_\tau(\breve{\Pi})$ result by applying the same procedure to $\Pi_0 = {\rm res}'_\tau(\breve{\Pi})$, hence they coincide. The theorem follows.
\end{proof}

\begin{Remark} We will use Theorem \ref{Knop_comp_thm} to make the link between the twisted Weyl character formula and the Lusztig character formula, in the proof of Theorem \ref{coeff_KL_poly}.
\end{Remark}

\subsection{Macdonald root system $\Sigma_1$}

Macdonald \cite{Mac} defined a root system $\Sigma_1 \subset \Sigma_0 \cup \frac{1}{2}\Sigma_0$ such that $\Sigma_1^{\rm red} = \Sigma_0$, by requiring $a/2 \in \Sigma_1$ for any $a \in \Sigma_0$ with $q_{a} \neq q_{a+1}$.  This plays a crucial role in harmonic analysis on $p$-adic groups which are not split (\cite[$\S3,4$]{Mac}).  The following is a corollary of Theorem \ref{Knop_comp_thm}, and when combined with Theorems \ref{Sigma_thm}, \ref{Sigma_0_thm}, it gives a way to understand $\Sigma_1$ directly in terms of Galois actions on the absolute roots $\Phi$.

\begin{cor} \label{Mac_cor} Let $G/F$ be quasi-split.  Then $\Sigma_0 \cup \widetilde{\Sigma}_0 = \Sigma_1$.  Consequently
$$
\widetilde{\Sigma}_0 = {\rm res}_\tau(\breve{\Sigma}) = \Sigma_{1, \rm red}.
$$
\end{cor}

\begin{proof}
We may assume $\Sigma_0$ is irreducible. Then $\Sigma_1$ is also irreducible, and $\Sigma_0$, $\widetilde{\Sigma}_0$, and $\Sigma_1$ all have the same Weyl group, $W_0$. We suppose $a \in \Sigma_0$. We need to show that $\frac{1}{2}a \in \widetilde{\Sigma}_0$ if and only if $\frac{1}{2}a \in \Sigma_1$.  Since $a$ is $W_0$-conjugate to an element in $\Pi_0$, we may assume $a \in \Pi_0$.  If $\frac{1}{2}a  \in \widetilde{\Sigma}$, then $a$ is special and Remark \ref{q_a_rmk} shows that $q_a \neq q_{a+1}$, hence $\frac{1}{2}a \in \Sigma_1$.  Conversely, if $\frac{1}{2}a \in \Sigma_1$, then $\Sigma_1$ is not reduced, so it is of type $BC_n$. Then $\Sigma_0$ is of type $C_n$, and $a$ is the long simple root.  By definition of $\Sigma_1$, we have $q_a \neq q_{a+1}$, and so again by Remark \ref{q_a_rmk}, we see $a$ is special, and $\frac{1}{2}a \in \widetilde{\Sigma}_0$.
\end{proof}

\section{The geometric basis and the stable Bernstein center} \label{geom+stable-center_sec}

Throughout this section, we retain the assumptions of $\S\ref{Sigma0_subsec}$: $G, B, S, T$ are defined and quasi-split over $F$. The case where $G$ is not necessarily quasi-split over $F$ will be considered elsewhere.

\subsection{The geometric basis for the center of a parahoric Hecke algebra} \label{geom_basis_sec}

In this section we construct a natural basis for the center of a parahoric Hecke algebra indexed by certain highest weight representations of $\widehat{G}^I$. We call this the {\em geometric basis}.


The complex dual group $\widehat{G}$ carries a $\mathcal W_F$-fixed splitting $(\widehat{B}, \widehat{T}, \widehat{X})$, and we set $^LG := \widehat{G} \rtimes \mathcal W_F$. Let $\widehat{G}^{I, \circ}$ denote the connected component of the reductive group $\widehat{G}^I$. We identify the Weyl group of $\widehat{G}$ with $W$, and then the Weyl group of $\widehat{G}^{I, \circ}$ is identified with $W^I = \breve{W}$ and the Weyl group of $\widehat{G}^{\mathcal W_F, \circ} = \widehat{G}^{I, \tau, \circ}$ is identified with $W_0 = (\breve{W})^{\langle \tau \rangle}$, the relative Weyl group of $G/F$.  See \cite[Prop.\,4.1]{H15}. 

Let $n$ denote the order of $\tau \in {\rm Aut}(\widehat{G}^I)$.  Let $\xi \in \mathbb C^\times$ be an $n$-th root of unity. 

If $\bar{\lambda} \in X^*(\widehat{T}^I)^{+, \tau}$, then $V_{\bar{\lambda}} \in {\rm Rep}(\widehat{G}^I)$ can be extended uniquely to a representation $(V_{\bar{\lambda}, \xi}, r_{\bar{\lambda}, \xi})$ of $\widehat{G}^I \rtimes \langle \tau \rangle$ such that $\tau$ acts by the scalar $\xi$  on the weight spaces associated to all ${\bar{\lambda}}' \in W_0\bar{\lambda}$. 

The irreducible $\widehat{T}^I \rtimes \langle \tau \rangle$-subrepresentations of $V_{\bar{\lambda},\xi}$ are of the form
$$
V_0 \oplus \tau(V_0) \oplus \cdots \oplus \tau^{d-1}(V_0)
$$
where $V_0$ is some 1-dimensional $\widehat{T}^I$-subrepresentation of weight $\bar{\nu} \in {\mathcal Wt}(\bar{\lambda})$, $d$ (a divisor of $n$) is the smallest element of $\mathbb N$ with $\tau^d(\bar{\nu}) = \bar{\nu}$ , and $\tau^d$ acts on $V_0$ by multiplication by some $\xi'_{n/d} \in {\mathbb \mu}_{n/d}(\mathbb C)$. 

\begin{defn} \label{Trep_defn}
Write this representation as $V_0(\bar{\nu}, d, \xi'_{n/d})$; write $m_{\bar{\lambda}, \xi}(\bar{\nu}, d, \xi'_{n/d})$ for its multiplicity in $V_{\bar{\lambda}, \xi}|_{\widehat{T}^I \rtimes \langle \tau \rangle}$.  If $d=1$, write it simply as $\bar{\nu} \boxtimes \xi'$, and write $m_{\bar{\lambda},\xi}(\bar{\nu}, \xi')$ for the multiplicity. 
\end{defn}
Thus
\begin{equation} \label{rest_eq}
V_{\bar{\lambda}, \xi}|_{\widehat{T}^I \rtimes \langle \tau \rangle} = \bigoplus_{\bar{\nu}, \xi'} m_{\bar{\lambda}, \xi}(\bar{\nu}, \xi') \, \bar{\nu} \boxtimes \xi'  ~ \oplus  ~ 
\bigoplus_{\underset{d > 1}{\bar{\nu},\,d,\,\xi'_{n/d}}} m_{\bar{\lambda}, \xi}(\bar{\nu}, d, \xi'_{n/d}) \, V_0(\bar{\nu}, d,\xi'_{n/d}).
\end{equation}
Here $\bar{\nu}$ ranges over ${\mathcal Wt}(\bar{\lambda})$. Note that an element $\bar{\nu} \in X^*(\widehat{T}^I)^{\tau}$ is $B$-dominant if and only if $\langle \alpha, \bar{\nu} \rangle \geq 0$ for all $\alpha \in \breve{\Phi}^{+}$, or equivalently, $\langle \alpha, \bar{\nu} \rangle \geq 0$ for all $\alpha \in \Phi^+_0$.

\begin{lemma} \label{W_mult_lem} Fix $\bar{\nu} \in {\mathcal Wt}(\bar{\lambda}) \cap X^*(\widehat{T}^I)^{+, \tau}$. Then:
\begin{enumerate}
\item[(1)] $\breve{W}\bar{\nu} \cap X^*(\widehat{T}^I)^{\tau} = W_0 \,\bar{\nu}$; and
\item[(2)] $m_{\bar{\lambda}, \xi}(\bar{\nu}, d, \xi'_{n/d}) = m_{\bar{\lambda}, \xi}(\bar{\nu}', d, \xi'_{n/d}) $ for all $\bar{\nu}' \in W_0 \bar{\nu}$.  
\end{enumerate}
\end{lemma}

\begin{proof}
Let $\bar{\nu}' \in \breve{W}\bar{\nu} \cap X^*(\widehat{T}^I)^\tau$ and choose $\bar{\nu}'' \in W_0 \, \bar{\nu}'$ which is $B$-dominant. Then $\bar{\nu}'' \in \breve{W}\bar{\nu}$ and is $B$-dominant, hence $\bar{\nu}'' = \bar{\nu}$.  This proves (1). For (2) we use the fact that elements of $W_0$ can be lifted to $\mathcal W_F$-fixed elements in $N_{\widehat{G}^{I, \circ}}(\widehat{T}^{I, \circ})$ which normalize $\widehat{T}^I$, by \cite[Prop.\,4.1]{H15}.
\end{proof}

Let $J \subset G(F)$ be any parahoric subgroup which corresponds to a $\tau$-stable facet in the apartment $X_*(S)\otimes \mathbb R$ of $\mathcal B(G, \breve{F})$, and let $\mathcal Z(G(F), J)$ be the center of the corresponding Hecke algebra. For $\bar{\nu} \in X^*(\widehat{T}^I)^{+, \tau}$, let $[\bar{\nu}] := \sum_{\bar{\nu}' \in W_0 \, \bar{\nu}} \bar{\nu}'$. Under the Bernstein isomorphism (\cite[Thm.\,11.9.1]{H14})
\begin{equation} \label{Bern_isom}
\mathbb C[X^*(\widehat{T}^I)^\tau]^{W_0} ~~ \overset{\sim}{\longrightarrow} ~~ \mathcal Z(G(F), J),
\end{equation}
the class $[\bar{\nu}]$ goes over to a function $z_{\bar{\nu}, J} \in \mathcal Z(G(F), J)$. By \cite[11.8]{H14}, it is characterized as follows: if $\pi$ is an irreducible admissible representation of $G(F)$ with $\pi^J \neq 0$, and if $s(\pi) \in (\widehat{T}^I)_{\langle \tau \rangle}/W_0$ is the Satake parameter attached to $\pi$ by \cite{H15}, then 
\begin{equation} \label{Bern_char}
z_{\bar{\nu}, J} \,\,\, \,\mbox{acts on $\pi^J$ by the scalar} \,\,\,\, {\rm tr}(s(\pi) \rtimes \tau \, | \, [\bar{\nu}] \boxtimes 1)
\end{equation}
where $[\bar{\nu}] \boxtimes 1 := \sum_{\bar{\nu}' \in W_0 \, \bar{\nu}} \bar{\nu}' \boxtimes 1$. (See Definition \ref{Trep_defn}.) It is known that $\{ z_{\bar{\lambda}, J} \}_{\bar{\lambda}}$ forms a basis for $\mathcal Z(G(F), J)$, as $\bar{\lambda}$ ranges over $X^*(\widehat{T}^I)^{+, \tau}$ (see \cite[Thm.\,11.10.1]{H14}).

For $\bar{\lambda} \in X^*(\widehat{T}^I)^{+, \tau}$, let ${\mathcal Wt}(\bar{\lambda})^+ = {\mathcal Wt}(\bar{\lambda}) \cap X^*(\widehat{T}^I)^+$, ${\mathcal Wt}(\bar{\lambda})^\tau = {\mathcal Wt}(\bar{\lambda}) \cap X^*(\widehat{T}^I)^\tau$, and ${\mathcal Wt}(\bar{\lambda})^{+, \tau} = {\mathcal Wt}(\bar{\lambda})^+ \cap {\mathcal Wt}(\bar{\lambda})^\tau$.

The operator $\tau \in {\rm Aut}(V_{\bar{\lambda},1})$ preserves the $\widehat{T}^I$-weight space $V_{\bar{\lambda}, 1}(\bar{\nu})$, since the weight $\bar{\nu}$ is $\tau$-fixed. 

\begin{defn} \label{C_defn}
For $\bar{\lambda} \in X^*(\widehat{T}^I)^{+, \tau}$, define the {\em geometric basis} element of $\mathcal Z(G(F), J)$ by setting
$$
C_{\bar{\lambda}, J} = \sum_{\underset{\xi'}{\bar{\nu} \in {\mathcal Wt}(\bar{\lambda})^{+, \tau}}} m_{\bar{\lambda}, 1}(\bar{\nu}, \xi') \, \xi' \, z_{\bar{\nu}, J} = \sum_{\bar{\nu} \in {\mathcal Wt}(\bar{\lambda})^{+, \tau}} {\rm tr}(\tau \, | \, V_{\bar{\lambda}, 1}(\bar{\nu}))\, z_{\bar{\nu}, J}.
$$
\end{defn}
Note that $m_{\bar{\lambda}, 1}(\bar{\lambda}, 1) = 1$ and $m_{\bar{\lambda}, 1}(\bar{\lambda}, \xi') = 0$ if $\xi' \neq 1$. This, together with Lemma \ref{W_mult_lem} and the above discussion, yields the following lemma. 

\begin{lemma} \label{geom_basis_lem}
The elements $\{ C_{\bar{\lambda}, J} \}_{\bar{\lambda}}$ form a basis for $\mathcal Z(G(F),J)$, characterized as follows: in the notation of \textup{(}\ref{Bern_char}\textup{)},
\begin{equation} \label{C_char}
C_{\bar{\lambda}, J} \,\,\, \,\mbox{acts on $\pi^J$ by the scalar} \,\,\,\, {\rm tr}(s(\pi) \rtimes \tau \, | \, V_{\bar{\lambda},1}).
\end{equation}
\end{lemma}
Note that in (\ref{rest_eq}), the terms with $d > 1$ make no contribution to ${\rm tr}(s(\pi) \rtimes \tau ~ | ~ V_{\bar{\lambda}, 1})$.

\subsection{Geometric basis in terms of Kazhdan-Lusztig polynomials}

Fix $\bar{\nu} \in {\mathcal Wt}(\bar{\lambda})^{+, \tau}$ as above. View $\bar{\lambda}$ and $\bar{\nu}$ as elements in $(X_*(T)_I)^\tau$.  These can be viewed as ``translation'' elements $t_{\bar{\lambda}}$, $t_{\bar{\nu}}$ in the extended affine Weyl group 
\begin{equation} \label{Wtilde^tau}
\widetilde{W}^\tau = (X_*(T)_I)^\tau \rtimes W_0 = \breve{W}^\tau_{\rm aff} \rtimes \Omega^\tau_{\breve{\bf a}}
\end{equation}
for the group $G/F$ (we are using the set-up of $\S\ref{Sigma0_subsec}$). Recall that the root system $\Sigma_0$ associated to the relative roots $\Phi_0$ for $G/F$ satisfies
$$
\breve{W}^\tau_{\rm aff} \cong \mathbb Z[\Sigma_0^\vee] \rtimes W_0.
$$
The quasi-Coxeter group structure on the right hand side of (\ref{Wtilde^tau}) is used to define the Bruhat order $\leq$ and the length function $\ell$ on $\widetilde{W}^\tau$ in the usual way.  For $x, y \in \widetilde{W}^\tau$ with the same $\Omega^\tau_{\bf a}$-component we can define the Kazhdan-Lusztig polynomials $P_{x,y}(q^{1/2}) \in \mathbb Z[q^{1/2}]$. Here we need to use the KL polynomials attached to the Coxeter group $\breve{W}^\tau_{\rm aff}$ and possibly unequal parameters; in our situation (see $\S\ref{Knop_subsec}$) we need a special case of the general theory of Lusztig \cite{Lus03}: the parameters are of the form $q^{L(w)/2}$, where  $L: \widetilde{W}^\tau \rightarrow \mathbb Z_{\geq 1}$ is defined as in $\S\ref{Knop_subsec}$, and satisfies the properties required by Lusztig \cite[3.1]{Lus03}, such as $L(w) > 0$ for $w \in \widetilde{W}^\tau$.  See also \cite{Kn}.

\begin{theorem} \label{coeff_KL_poly} For $\bar{\lambda}  \in (X_*(T)_I)^\tau$, let $w_{\bar{\lambda}} \in W_0 t_{\bar{\lambda}} W_0$ be the unique element of maximal length in this double coset. Then for $\bar{\nu} \in {\mathcal Wt}(\bar{\lambda})^{+, \tau}$,
$$
{\rm tr}(\tau \, | \, V_{\bar{\lambda}, 1}(\bar{\nu})) = P_{w_{\bar{\nu}}, w_{\bar{\lambda}}} (1), 
$$
and thus
$$
C_{\bar{\lambda}, J} = \sum_{\bar{\nu} \in {\mathcal Wt}(\bar{\lambda})^{+, \tau}} P_{w_{\bar{\nu}}, w_{\bar{\lambda}}} (1) \, z_{\bar{\nu}, J}.
$$
\end{theorem}
We will prove Theorem \ref{coeff_KL_poly} in subsection \ref{KL_poly_pf_sec}.

\begin{Remark}
 Computing $C_{\bar{\lambda}, J}$ explicitly amounts to computing explicitly the coefficients ${\rm tr}(\tau \, | \, V_{\bar{\lambda},1}(\bar{\nu})) = P_{w_{\bar{\nu}}, w_{\bar{\lambda}}}(1)$ and the elements $z_{\bar{\nu}, J}$. There are well-known algorithms for computing $P_{w_{\bar{\nu}}, w_{\bar{\lambda}}}(1)$. As for the $z_{\bar{\nu}, J}$, the main problem is to compute them explicitly when $J = I$ is an Iwahori subgroup of $G(F)$, in terms of the Iwahori-Matsumoto generators $T_w \,\, (w \in \widetilde{W}^\tau)$ for the Iwahori-Hecke algebra $\mathcal H(G(F), I)$. This can be done using the theory of alcove-walks, see \cite{Gor07} and \cite[Appendix]{HR12}. Thus, in principle, the geometric basis elements $C_{\bar{\lambda},J}$ can be computed explicitly, in every case.
\end{Remark}

\subsection{Proof of Theorem \ref{coeff_KL_poly}} \label{KL_poly_pf_sec}


The proof relies on the twisted Weyl character formula (a.k.a.\,the Twining Character Formula of Jantzen \cite{Jan}) and Lusztig's character formula for unequal parameters \cite[Cor.\,6.4]{Kn}. Theorem \ref{Knop_comp_thm} is what allows us to relate these two formulas. They apply to {\em connected} reductive complex groups with a splitting-preserving automorphism. Using Lemma \ref{ht_wt_lem}(4), we may replace $\bar{\lambda}, \, \bar{\nu}$ with $\bar{\lambda}^\flat, \, \bar{\nu}^\flat$ and work in the connected reductive group $\widehat{G}^{I, \circ}$ endowed with the automorphism $\tau$.  Note that $P_{w_{\bar{\nu}} , w_{\bar{\lambda}}}(q^{1/2}) = P_{w_{\bar{\nu}^\flat}, w_{\bar{\lambda}^\flat}}(q^{1/2})$ because the torsion subgroup of $X_*(T)_I$ lies the center of $\widetilde{W}$ and in particular in $\Omega_{\breve{\bf a}}$; see \cite[$\S 8.1$]{HH}.

\subsubsection{The twisted Weyl character formula}

We review the twisted Weyl character formula; a statement can be found in \cite[Prop.\,5.1]{Chr}. This is equivalent to the so-called Twining Character Formula of Jantzen \cite[Satz 9]{Jan}. 

Use the symbol $e^{\bar{\nu}^\flat}$ to denote the element $\bar{\nu}^\flat \in X^*(\widehat{T}^{I, \circ})$ when it is viewed in the group algebra $\mathbb C[X^*(\widehat{T}^{I, \circ})]$, so that we have $e^{\bar{\nu}} e^{\bar{\nu}'} = e^{\bar{\nu} + \bar{\nu}'}$ and $w(e^{\bar{\nu}}) = e^{w(\bar{\nu})}$. Then the twisted Weyl character formula may be stated as follows. 

\begin{theorem} [Twisted Weyl character formula - original form] \label{twisted_Weyl_0} For $\bar{\lambda}^\flat \in X^*(\widehat{T}^{I,\circ})^{+,\tau}$, there is an equality
\begin{equation} \label{Jant_form_0}
\sum_{\bar{\nu}^\flat \in {\mathcal Wt}(\bar{\lambda}^\flat)^\tau} {\rm tr}(\tau \, | \, V_{\bar{\lambda}^\flat, 1}(\bar{\nu}^\flat)) \, e^{\bar{\nu}^\flat}= \sum_{w \in W_0} w \Big( \prod_{\alpha \in N'_\tau(\breve{\Sigma}^\vee)^+}\frac{1}{1 - e^{-\alpha}} \, \Big) \cdot e^{w\bar{\lambda}^\flat}.
\end{equation}
\end{theorem}

\begin{proof}
See for instance \,\cite[Prop.\,5.1]{Chr}.
\end{proof}
It is clear that the right hand side is the character of a highest weight representation of a suitable connected reductive group (since Propositions \ref{IMRN_prop}, \ref{on_Sigma}, and \ref{dual_prop} show that $N'_\tau(\breve{\Sigma}^\vee)$ is the set of roots of $\widehat{(\widehat{H^\circ}^\tau)}$, where $H^\circ := \widehat{G}^{I, \circ}$).

\begin{Remark}
All of the references known to the author (e.g.\,\cite{Jan, Chr, KLP, Hong}, and the references mentioned in \cite{KLP, Hong}) work with a simply connected complex group instead of a connected reductive group such as $\widehat{G}^{I, \circ}$. One way to give a conceptual proof in the general connected reductive case is to adapt the argument of Hong \cite{Hong}, who proved Jantzen's Twining Character Formula for simply connected complex groups, using the geometric Satake correspondence. 
\end{Remark}

Using Theorem \ref{Knop_comp_thm}, we can rewrite (\ref{Jant_form_0}) in a form which is more convenient for us.

\begin{theorem} [Twisted Weyl character formula] \label{twisted_Weyl} For $\bar{\lambda}^\flat \in X^*(\widehat{T}^{I,\circ})^{+,\tau}$, there is an equality
\begin{equation} \label{Jant_form}
\sum_{\bar{\nu}^\flat \in {\mathcal Wt}(\bar{\lambda}^\flat)^\tau} {\rm tr}(\tau \, | \, V_{\bar{\lambda}^\flat, 1}(\bar{\nu}^\flat)) \, e^{\bar{\nu}^\flat}= \sum_{w \in W_0} w \Big( \prod_{\alpha \in \widetilde{\Sigma}_0^+}\frac{1}{1 - e^{-\alpha^\vee}} \, \Big) \cdot e^{w\bar{\lambda}^\flat}.
\end{equation}
\end{theorem}

 We write the twisted Weyl character formula as above in order to make the connection with Knop's version of the Lusztig character formula below.

\subsubsection{Analogue of Lusztig's character formula}

Combined with Theorem \ref{twisted_Weyl}, the following formula immediately implies Theorem \ref{coeff_KL_poly}.

\begin{theorem}[Knop] \label{Lus_char_form} The following analogue of Lusztig's character formula from \cite[Thm.\,6.1]{Lus83} holds \textup{:}
\begin{equation} \label{Lus_char_form_eq}
\sum_{w \in W_0} w \Big( \prod_{\alpha \in \widetilde{\Sigma}_0^+}\frac{1}{1 - e^{-\alpha^\vee}} \, \Big) \cdot e^{w\bar{\lambda}^\flat} = \sum_{\bar{\nu}^\flat \in {\mathcal Wt}(\bar{\lambda}^\flat)^\tau} P_{w_{\bar{\nu}^\flat}, w_{\bar{\lambda}^\flat}}(1) \cdot  e^{\bar{\nu}^\flat}.
\end{equation}
\end{theorem}

\begin{proof}
This follows from \cite[Cor.\,6.4]{Kn}, for which we sketch the proof. By the Weyl character formula, the left hand side is the character of an irreducible representation of a complex group with root system $N'_\tau(\breve{\Sigma}^\vee) = \widetilde{\Sigma}_0^\vee$ and highest weight $\bar{\lambda}^\flat$. By the Demazure character formula (cf.\,e.g.\,\cite{A}), the left hand side is also $s_{\bar{\lambda}^\flat} := \Delta_{w_0}(e^{\bar{\lambda}^\flat})$ where $\Delta_{w_0}$ is the Demazure operator for the longest element $w_0 \in W_0$, relative to the root system $\widetilde{\Sigma}_0$ (\cite[p.\,422]{Kn}).  Let $\Psi$ be the ``Satake isomorphism'' of \cite[(5.2)]{Kn}. In \cite[Thm.\,6.2]{Kn}, Knop shows that $\Psi(s_{\bar{\lambda}^\flat})$ is a {\em KL element in ${\bf H}(\widetilde{W}^\tau, S^\tau_{\rm aff}, L)$ for $w_{\bar{\lambda}^\flat}$} in the sense of  \cite[Def.\,p.\,424]{Kn}.  This together with \cite[Thm.\,5.1, (6.1)]{Kn} shows that 
\begin{equation} \label{LK_eq}
\Psi(s_{\bar{\lambda}^\flat}) = \sum_{\bar{\nu}^\flat \in {\mathcal Wt}(\bar{\lambda}^\flat)^{+, \tau}} P_{w_{\bar{\nu}^\flat}, w_{\bar{\lambda}^\flat}}(q^{1/2}) \, N_{\bar{\nu}^\flat},
\end{equation}
where, up to an integral power of $q^{1/2}$, the element $N_{\bar{\nu}^\flat}$ (\cite[(5.3)]{Kn}) is the standard basis of the spherical affine Hecke algebra ${\bf H}^{\rm sph}$ (\cite[(5.1)]{Kn}). This is therefore an analogue of the Lusztig-Kato formula for quasi-split groups (a proof for split groups can be found in \cite{Ka} or \cite[Thm.\,7.8.1]{HKP}). The result now follows by applying $\Psi^{-1}$ to (\ref{LK_eq}) and then specializing $q^{1/2} \mapsto 1$. 
\end{proof}

\subsubsection{Conclusion of proof of Theorem \ref{coeff_KL_poly}}

Theorem \ref{coeff_KL_poly} follows immediately by combining Theorems \ref{twisted_Weyl} and \ref{Lus_char_form}. \qed

\subsection{Some elements in the stable Bernstein center} \label{some_elts_sec}

We assume from now on that the conjugacy class $\{ \mu \} \subset X_*(T)$ is defined over $F$. This is enough for the applications to Shimura varieties which will be explained in section \ref{Shim_var_sec}. Let $\mu \in \{\mu \}$ be $B$-dominant; it is therefore fixed by $\mathcal W_F$.  According to \cite[2.1.2]{Ko84}, there is a unique way to extend $V_\mu \in {\rm Rep}(\widehat{G})$ to a representation $(V_\mu, r_\mu)$ of $^LG$, such that $\mathcal W_F$ acts trivially on the weight spaces associated to all $\mu' \in W_0\mu$.

Our goal is to describe certain elements of the stable Bernstein center in terms of the basis elements $C_{\bar{\lambda}, J}$.  The functions we want to understand are denoted $Z_{V} * 1_J$, where $V \in {\rm Rep}(\,^LG)$. The $Z_{V}$ and are called elements of the {\em geometric Bernstein center} in \cite[Def.\,6.2.1]{H14}; the functions $Z_{V_{\mu}} * 1_J$ play an important role in conjectures about cohomology of Shimura varieties $Sh_{K^p J}({\bf G}, \{h\})$ with level structure at $p$ given by a parahoric subgroup $J \subset {\bf G}(\mathbb Q_p)$; see especially the conjectures related to quasi-split groups in \cite[$\S7.2$]{H14}, and section \ref{Shim_var_sec} below.

Let us briefly recall the construction of $Z_V$ following \cite{H14}. As in \cite[$\S5.7$]{H14}, associated to a representation $(V, r)$ of $^LG = \widehat{G} \rtimes \mathcal W_F$ we get an element $Z_{V}$ in the stable Bernstein center of $G/F$. Assuming $G$ satisfies the enhanced local Langlands correspondence LLC+ (see \cite[5.2.1]{H14}), we obtain an element in the actual Bernstein center, also denoted by $Z_{V}$. We view $Z_{V}$ as a $G(F)$-invariant essentially compact distribution on $G(F)$. Then we get $Z_{V} * 1_J \in \mathcal Z(G(F), J)$ (here $1_J$ is the identity element of $\mathcal Z(G(F), J)$). It is characterized by
\begin{equation} \label{char2}
Z_{V} * 1_J \,\,\,\, \mbox{acts on $\pi^J$ by the scalar} \,\,\,\, {\rm tr}(\varphi_\pi(\tau) \, | \, V^{r \varphi_\pi(I)}_{\mu}),
\end{equation}
where $\varphi_\pi : \mathcal W_F \rightarrow \, ^LG$ is the local Langlands parameter associated to $\pi$. We assume from now on that when $\pi^J \neq 0$,
\begin{align*}
\varphi_\pi(\tau) &= s(\pi) \rtimes \tau \\
\varphi_\pi(\gamma) & = 1 \rtimes \gamma
\end{align*}
for all $\gamma \in I$ (comp.\,\cite[Conj.\,13.1]{H15}). (Of course the existence and properties of $\varphi_\pi$ for such $\pi$ follow from the Deligne-Langlands correspondence, which is known for many $p$-adic groups, e.g.\,\cite{KL}. See also \cite[beginning of $\S7$]{H14} for the proof of these properties for general quasi-split groups $G$, assuming LLC+ holds for $G$.) Our assumptions then imply that
\begin{equation} \label{char3}
Z_{V} * 1_J \,\,\,\, \mbox{acts on $\pi^J$ by the scalar} \,\,\,\, {\rm tr}(s(\pi) \rtimes \tau \, | \, V^I),
\end{equation}
where $V^I := V^{r_\mu(1 \rtimes I)}$. For our purposes, we could also {\em define} $Z_V * 1_J$ to be the unique element of $\mathcal Z(G(F), J)$ satisfying (\ref{char3}); this would be an unconditional definition, avoiding the assumption that $G$ satisfies LLC+. Let us prove that such a function exists unconditionally. Let $T(F)_1 = {\rm ker}(\kappa_T: T(F) \rightarrow X_*(T)_{I_F}^{\Phi_F})$ be the kernel of the Kottwitz homomorphism and consider the set ${\rm X}^{\rm w}(T)$ of weakly unramified characters $\chi: T(F) \rightarrow \mathbb C^\times$ (i.e.\,characters trivial on $T(F)_1$). By the Kottwitz isomorphism, we can view $\chi \in {\rm X}^{\rm w}(T)$ as an element of the complex diagonalizable group $(\widehat{T}^{I_F})_\tau$. If $\pi$ has supercuspidal support $(T, \chi)_G$, then $s(\pi) \rtimes \tau = \chi \rtimes \tau$. The function $\chi \mapsto {\rm tr}(\chi \rtimes \tau \, | \, V^I)$ is clearly a regular function on the quotient variety $(\widehat{T}^{I_F})_\tau/W_0$, which is identified with the variety of supercuspidal supports of representations with $J$-fixed vectors.  By the theory of the Bernstein center, this regular function gives rise to the desired unconditional function in $\mathcal Z(G(F), J)$.

Now we return to our $F$-rational conjugacy class $\{ \mu \}$. 
Now $V_\mu|_{\widehat{G}^I}$ is a representation of $\widehat{G}^I \rtimes \mathcal W_F$ and of $\widehat{G}^I \times I$; decomposing into isotypical components for the various irreducible representations $V_{\bar{\lambda}} \in {\rm Rep}(\widehat{G}^I)$, and then looking at the contribution of the trivial representation of $I$, we obtain $V^I_\mu|_{\widehat{G}^I}$ as a sum of certain $V_{\bar{\lambda}}$.  Let ${\mathcal Wt}(\bar{\mu})$ denote the set of $\widehat{T}^I$-weights in $V_{\bar{\mu}}$. The proof of Theorem \ref{PRS_thm} shows that
$$
\overline{\mathcal Wt(\mu)} \subseteq {\mathcal Wt}(\bar{\mu})
$$
and $\bar{\mu}$ is the unique $\preceq$-maximal element in both sets; in fact it also shows that $\lambda \prec \mu$ implies $\bar{\lambda} \prec \bar{\mu}$, and thus the weight $\bar{\mu}$ appears in $V^I_\mu|_{\widehat{G}^I}$ with multiplicity one. Therefore, we get a decomposition in ${\rm Rep}(\widehat{G}^I)$
\begin{equation} \label{decomp}
V^I_\mu|_{\widehat{G}^I} = V_{\bar{\mu}} ~~ \oplus ~~  \bigoplus_{\bar{\lambda} \prec \bar{\mu}} a_{\bar{\lambda}, \mu} \, V_{\bar{\lambda}}
\end{equation}
where $a_{\bar{\lambda}, \mu} \in \mathbb Z_{\geq 0}$ and $\bar{\lambda}$ ranges over elements in the finite set $\overline{{\mathcal Wt}(\mu)} \cap X^*(\widehat{T}^I)^+$.  Since every $\bar{\lambda} \in {\mathcal Wt}(\bar{\mu})$ appears as a $\widehat{T}^I$-weight in $V_{\bar{\mu}}$, it also appears as a weight in $V_\mu^I|_{\widehat{G}^I}$. It follows that we actually have equality
\begin{equation} \label{weight_equality}
\overline{{\mathcal Wt}(\mu)} = {\mathcal Wt}(\bar{\mu}).
\end{equation}

\begin{Remark}
Because $V_{\bar{\mu}}$ appears with multiplicity one and $\bar{\mu}$ is the unique maximal element, the integers $a_{\bar{\lambda}, \mu}$ may be computed recursively in practice.
\end{Remark}

Now, $V^I_\mu|_{\widehat{G}^I}$ and $V_{\bar{\mu}}$ both extend to representations of $\widehat{G}^I \rtimes \langle \tau \rangle$. The remaining summands in (\ref{decomp}) are not necessarily stable under $\tau$, but may be permuted. But if we calculate the trace of an element of the form $s(\pi) \rtimes \tau$, only those $\bar{\lambda}$ which are $\tau$-fixed will contribute, and for those, we can regard, as above, $V_{\bar{\lambda}}$ as the representation $V_{\bar{\lambda},1}$ of $\widehat{G}^I \rtimes \langle \tau \rangle$.

Suppose $\bar{\lambda} \in {\mathcal Wt}(\bar{\mu})^+$. Let $V^I_{\mu}(\bar{\lambda})$ denote the sum of all $\widehat{G}^I$-submodules of $V^I_{\mu}$ which are $\widehat{G}^I$-isomorphic to $V_{\bar{\lambda}}$. 

If $\bar{\lambda} \in {\mathcal Wt}(\bar{\mu})^{+, \tau}$, we define the space ${\mathbb H}_\mu(\bar{\lambda})$ of ``vectors with highest weight $\bar{\lambda}$'' appearing in (\ref{decomp}). Namely, let ${\mathbb H}_\mu(\bar{\lambda}) \subset V^I_{\mu}$ be the set of all vectors $v'$ such that $\widehat{T}^{I}$ acts on $v'$ through the character $\bar{\lambda}$, and $v'$ is killed by ${\rm Lie}(\widehat{U}^{I,\circ})$. Then as $\bar{\lambda}$ is $\tau$-fixed and $\widehat{T}^I$ and $\widehat{U}^{I,\circ}$ are $\tau$-stable, ${\mathbb H}_{\mu}(\bar{\lambda})$ carries an action by $\tau$, which may be diagonalized. Thus, we have a $\widehat{G}^I \rtimes \langle \tau \rangle$ isomorphism
\begin{equation} \label{rep_factoring}
V^I_{\mu}(\bar{\lambda}) \cong V_{\bar{\lambda},1} \otimes  {\mathbb H}_\mu(\bar{\lambda}),
\end{equation}
where $g \rtimes \tau^i$ acts on the r.h.s.\,by the rule $v \otimes v' \mapsto (g\tau^i)v \otimes \tau^i v'$.

Clearly we have isomorphisms of $\widehat{G}^I \rtimes \langle \tau \rangle$-modules
\begin{equation} \label{V^I_full}
V_\mu^I|_{\widehat{G}^I \rtimes \langle \tau \rangle} = \bigoplus_{\bar{\lambda} \in {\mathcal Wt}(\bar{\mu})^{+, \tau}} V_{\bar{\lambda}, 1} \otimes {\mathbb H}_\mu(\bar{\lambda})  ~\oplus ~ \bigoplus_{\underset{\tau\bar{\lambda} \neq \bar{\lambda}}{\bar{\lambda} \in {\mathcal Wt}(\bar{\mu})^+}} V^I_\mu(\bar{\lambda}).
\end{equation}  
In particular, for $\bar{\lambda} \in {\mathcal Wt}(\bar{\mu})^{+, \tau}$, $${\rm dim}(\mathbb H_\mu(\bar{\lambda})) = a_{\bar{\lambda}, \mu}.$$

Comparing ${\rm tr}(s(\pi) \rtimes \tau \, | \, \cdot)$ on both sides of $(\ref{V^I_full})$, it is clear that (\ref{C_char}) and (\ref{char3}) imply the following theorem.

\begin{theorem} \label{geom_center_thm} In the situation above,
$$
Z_{V_{\mu}} * 1_J = \sum_{\bar{\lambda} \in {\mathcal Wt}(\bar{\mu})^{+, \tau}} {\rm tr}(\tau \, | \, {\mathbb H}_\mu(\bar{\lambda})) \, C_{\bar{\lambda} , J}.
$$
\end{theorem}
Note that ${\rm tr}(\tau \, | \, {\mathbb H}_\mu(\bar{\mu})) = 1$ as ${\mathbb H}_{\mu}(\bar{\mu})$ is the trivial 1-dimensional representation of $\tau$.

\section{On test functions for quasi-split groups and parahoric level} \label{Shim_var_sec}

In \cite{H14} several conjectures are announced about test functions for Shimura varieties. We refer to \cite{H14} for details on what is meant by ``test function'' and for the statements of the conjectures in the general case. Here, we content ourselves to relate the previous sections of this paper to the conjectural test functions for Shimura varieties coming from Shimura data $({\bf G}/\mathbb Q, \{h \}, K^pK_p)$, where $K_p \subset {\bf G}(\mathbb Q_p)$ is a parahoric subgroup, and where ${\bf G}$ is quasi-split over the relevant extension of the local reflex field.  

We will need the following general lemma. The proof is straightforward and will be left to the reader.

\begin{lemma} \label{ramified_descent}
Let $E/F$ be a totally ramified extension of degree $n$. Let $G$ be a reductive group over $F$. Let $\mathcal W_E = I_E \rtimes \langle \tau \rangle$ and $\mathcal W_F = I_F \rtimes \langle \tau \rangle$ be the corresponding Weil groups. Suppose $e=h_1, h_2, \dots, h_n \in I_F$ is a fixed set of representatives of $I_F/I_E$ \textup{(}thus also of $\widehat{G} \rtimes \mathcal W_F/\widehat{G} \rtimes \mathcal W_E$\textup{)}. Let $V \in {\rm Rep}(\widehat{G} \rtimes \mathcal W_E)$, and consider the induced representation in ${\rm Rep}(\widehat{G} \rtimes \mathcal W_F)$ defined by
$$
I(V) = \mathbb Z[\widehat{G} \rtimes \mathcal W_F] \otimes_{\mathbb Z[\widehat{G} \rtimes \mathcal W_E]} V.
$$ 
Then there is an isomorphism of $\widehat{G}^{I_F} \rtimes \langle \tau \rangle$-modules
\begin{align} \label{ram_desc_eq}
V^{I_E} &\overset{\sim}{\longrightarrow} I(V)^{I_F} \\
v &\longmapsto \sum_{i=1}^n h_i \otimes v. \notag
\end{align}
\end{lemma}

We need to recall some notation from \cite[$\S6$]{H14}. Write $G := {\bf G}_{\mathbb Q_p}$, and suppose $E/\mathbb Q_p$ is the local reflex field, that is, the field of definition of the geometric conjugacy class $\{ \mu \} \subset X_*(G_{\bar{\mathbb Q}_p})$, where $\mu = \mu_h$ is attached to $h$ in the usual way.

Let $E_j/E$ be an unramified extension of degree $j \geq 1$, let $E_0/\mathbb Q_p$ be the maximal unramified subextension of $E/\mathbb Q_p$. Let $E_{j0}/\mathbb Q_p$ be the maximal unramified subextension of $E_j/\mathbb Q_p$. Then $E/E_0$ and $E_j/E_{j0}$ are totally ramified of the same degree, and $E_{j0} = \mathbb Q_{p^r}$, where $r := j[E_0: \mathbb Q_p]$.  Thus we have a tower of fields
$$
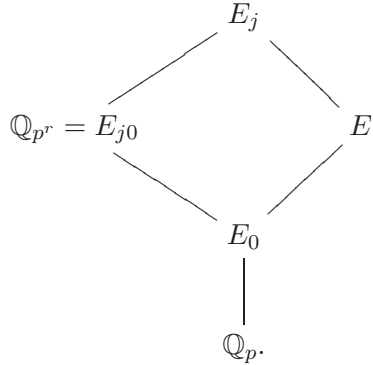
\xymatrix{
& E_{j} \ar@{-}[dl]  \ar@{-}[dr] & \\
\mathbb Q_{p^r} = E_{j0}  \ar@{-}[dr] & & \ar@{-}[dl] E \\
& E_0 \ar@{-}[d] & \\
& \mathbb Q_p. & }
$$
As in \cite[$\S6$]{H14}, we have a representation $V^E_\mu$ of $\widehat{G} \rtimes \mathcal W_{E}$, its restriction $V^{E_{j}}_{\mu, j}$ to $\widehat{G} \rtimes \mathcal W_{E_j}$, and the induction of the latter to $\widehat{G} \rtimes \mathcal W_{E_{j0}}$
$$
V^{E_{j0}}_{\mu, j} := {\rm Ind}^{\widehat{G} \rtimes \mathcal W_{E_{j0}}}_{\widehat{G} \rtimes \mathcal W_{E_j}} V^{E_j}_{\mu, j}.
$$

{\bf Now assume that $G$ is quasi-split over $E_{j0} = \mathbb Q_{p^r}$.} The conjugacy class $\{ \mu \}$ can be represented by an $E_j$-rational element $\mu$, and this allows us to apply the material of $\S\ref{some_elts_sec}$ to $V^{E_{j0}}_{\mu, j}$. 

Let $J_{E_{j0}} \subset G(E_{j0})$ be the parahoric subgroup corresponding to $K_p \subset G(\mathbb Q_p)$.  The {\em Test Function Conjecture}, see \cite[$\S7.2$]{H14}, highlights the importance of the function
\begin{equation} \label{qs_deal}
 Z_{V^{E_{j0}}_{\mu, j}} * 1_{J_{E_{j0}}} \in \mathcal Z(G(E_{j0}), J_{E_{j0}}).
\end{equation}
\begin{theorem} \label{test_fcn_thm}
When $G/E_{j0}$ is quasi-split,
\begin{equation} \label{tf_exp_eq}
Z_{V^{E_{j0}}_{\mu, j}} * 1_{J_{E_{j0}}} = \sum_{\bar{\lambda} \in {\mathcal Wt}(\bar{\mu})^{+, \tau}_{E_j}} {\rm tr}(\tau \, | \, \mathbb H_{\mu, E_j}(\bar{\lambda})) \, \sum_{\bar{\nu} \in {\mathcal Wt}(\bar{\lambda})^{+, \tau}_{E_j}} P_{w_{\bar{\nu}}, w_{\bar{\lambda}}} (1) \, \Big(\sum_{\{\bar{\nu}' \} \subset W_{E_j}\bar{\nu}} z_{\bar{\bar{\nu}}'} \Big).
\end{equation}
Here the subscripts $E_j$ and $E_{j0}$ indicate objects attached to the appropriate group $G_{E_j}$ or $G_{E_{j0}}$. 
where $\{ \bar{\nu}' \}$ ranges over $W_{E_{j0}}$-conjugacy classes in $W_{E_j}\bar{\nu}$, and $\bar{\bar{\nu}}' \in X^*(\widehat{T}^{I_{E_{j0}}})^\tau$ is the image of $\bar{\nu}' \in X^*(\widehat{T}^{I_{E_j}})^\tau$. 
\end{theorem}
\begin{proof}
This follows from the proof of Theorem \ref{geom_center_thm} with $F = E_j$ and with $V_\mu$ replaced by $V^{E_j}_{\mu, j}$, combined with Lemma \ref{ramified_descent} for $F = E_{j0}$ and $E = E_j$. 
\end{proof}

To compute test functions when $G/F = G/\mathbb Q_{p^r}$ is not quasi-split, one should simply take the image of the suitable analogue of (\ref{qs_deal}) for a quasi-split inner form $G^*/F$ under a normalized transfer homomorphism. We refer to \cite[$\S7.3$]{H14} for details. 

Recently, Kisin and Pappas \cite{KiPa} constructed local models attached to abelian type Shimura varieties with parahoric level structure at a prime $p > 2$. The above formula (\ref{tf_exp_eq}) for the test function should play a role in describing the corresponding local Hasse-Weil zeta functions in terms of automorphic forms: the nearby cycles considered in \cite[$\S4.7$]{KiPa} should be compared to (\ref{tf_exp_eq}) when pursuing the Langlands-Kottwitz method, as for example in \cite{HR12}. These formulas should also be compared with the explicit results on nearby cycles in \cite{PR03, PR05}, and \cite{Lev}.\footnote{This comparison has been accomplished in the preprint \cite{HaRi}.}

\bigskip

\small
\bigskip
\obeylines
\noindent
University of Maryland
Department of Mathematics
College Park, MD 20742-4015 U.S.A.
email: tjh@math.umd.edu

\end{document}